\documentclass[11pt, a4paper]{amsart}
\usepackage{verbatim}
\usepackage[dvips]{color}
\usepackage{amssymb}
\usepackage[active]{srcltx}
\usepackage{latexsym}
\usepackage{amsmath}
 \usepackage{float}
\usepackage{mathrsfs} 
\usepackage[draft]{fixme}
\usepackage{graphicx, epstopdf, float}

\usepackage{enumitem}

\usepackage{latexsym}

\usepackage{color}

\allowdisplaybreaks

\newtheorem{theorem}{Theorem}[section]
\newtheorem{proposition}{Proposition}[section]
\newtheorem{lemma}{Lemma}[section]
\newtheorem{corollary}{Corollary}[section]
\theoremstyle{definition}
\newtheorem{definition}{Definition}
\newtheorem{remark}{Remark}[section]

\newtheorem{assumption}{Assumption}

\floatstyle{ruled} \restylefloat{figure} \restylefloat{table}

\numberwithin{equation}{section}

\newcommand{\beq}{\begin{equation}}
\newcommand{\bea}[1]{\begin{array}{#1} }
\newcommand{\eeq}{ \end{equation}}
\newcommand{\ea}{ \end{array}}

\DeclareMathOperator{\sign}{sign}

\def\mean#1{\mathchoice%
          {\mathop{\kern 0.2em\vrule width 0.6em height 0.69678ex depth -0.58065ex
                  \kern -0.8em \intop}\nolimits_{\kern -0.4em#1}}%
          {\mathop{\kern 0.1em\vrule width 0.5em height 0.69678ex depth -0.60387ex
                  \kern -0.6em \intop}\nolimits_{#1}}%
          {\mathop{\kern 0.1em\vrule width 0.5em height 0.69678ex
              depth -0.60387ex
                  \kern -0.6em \intop}\nolimits_{#1}}%
          {\mathop{\kern 0.1em\vrule width 0.5em height 0.69678ex depth -0.60387ex
                  \kern -0.6em \intop}\nolimits_{#1}}}

\def\vintslides_#1{\mathchoice%
          {\mathop{\kern 0.1em\vrule width 0.5em height 0.697ex depth -0.581ex
                  \kern -0.6em \intop}\nolimits_{\kern -0.4em#1}}%
          {\mathop{\kern 0.1em\vrule width 0.3em height 0.697ex depth -0.604ex
                  \kern -0.4em \intop}\nolimits_{#1}}%
          {\mathop{\kern 0.1em\vrule width 0.3em height 0.697ex depth -0.604ex
                  \kern -0.4em \intop}\nolimits_{#1}}%
          {\mathop{\kern 0.1em\vrule width 0.3em height 0.697ex depth -0.604ex
                  \kern -0.4em \intop}\nolimits_{#1}}}

\newcommand{\aveint}[2]{\mathchoice%
          {\mathop{\kern 0.2em\vrule width 0.6em height 0.69678ex depth -0.58065ex
                  \kern -0.8em \intop}\nolimits_{\kern -0.45em#1}^{#2}}%
          {\mathop{\kern 0.1em\vrule width 0.5em height 0.69678ex depth -0.60387ex
                  \kern -0.6em \intop}\nolimits_{#1}^{#2}}%
          {\mathop{\kern 0.1em\vrule width 0.5em height 0.69678ex depth -0.60387ex
                  \kern -0.6em \intop}\nolimits_{#1}^{#2}}%
          {\mathop{\kern 0.1em\vrule width 0.5em height 0.69678ex depth -0.60387ex
                  \kern -0.6em \intop}\nolimits_{#1}^{#2}}}

\def\eqn#1$$#2$${\begin{equation} \label#1#2\end{equation}}
\def\charfn_#1{{\raise1.2pt\hbox{$\chi
_{\kern-1pt\lower3pt\hbox{{$\scriptstyle#1$}}}$}}}

\def\qq1{q_*}
\def\q2{q_{**}}

\newdimen\vintbar
\def\vint{-\kern-\vintbar\int}

\def\0{\boldsymbol 0}

\newtoks\by
\newtoks\paper
\newtoks\book
\newtoks\jour
\newtoks\yr
\newtoks\pages
\newtoks\vol
\newtoks\publ

\def\name[#1, #2]{#1 #2}
\def\ota{{\hbox{\bf ???}}}
\def\cLear{\by = \ota\paper = \ota\book = \ota\jour = \ota\yr = \ota
\pages = \ota\vol = \ota\publ = \ota}
\def\endpaper{\the\by, \textit{\the\paper},
{\the\jour} \textbf{\the\vol} (\the\yr), \the\pages.\cLear}
\def\endbook{\the\by, \textit{\the\book},
\the\publ, \the\yr.\cLear}
\def\endpap{\the\by, \textit{\the\paper}, \the\jour.\cLear}
\def\endproc{\the\by, \textit{\the\paper}, \the\book, \the\publ,
\the\yr, \the\pages.\cLear}

\begin{document}

\title[Distribution tail asymptotics for one dimensional SDE]{On the distribution tail of stochastic differential equations: The one-dimensional case}

\author{ Sidi Mohamed Aly }
\address{Sidi Mohamed Aly\\ Centre for Mathematical Sciences, Mathematical Statistics\\
Lund University, Box 118, SE-221 00 Lund, Sweden}
\email{souldaly@gmail.com} %sidi@math.lth.se}

\begin{abstract}
%In a previous article (\cite{Aly13}) the author obtained a Tauberian result  which relates explicitly the moment generating function (MGF) and the complementary cumulative distribution function (CCDF) of a random variable whose MGF is known and finite only on part of the real line. This paper derives a similar result without assuming the knowledge of the explicit expression of the MGF.  Instead an upper bound and lower bound for the MGF as assumed to be known.
This paper considers a general one-dimensional stochastic differential equation (SDE). A particular attention is given to the SDEs that may be transformed (via Ito's formula) into:\begin{equation}\label{SDE-nonlinear-Drift}
d X_t = ( \bar{B} (X_t)  - b X_t) d t + \sqrt{X_t} d W_t, ~~~X_0 > 0,
\end{equation}
where $ \bar{B}(y)/ y \to 0$. It is shown that the MGF of $X_t$ explodes at  a critical moment $\mu^\ast_t$ which is independent of $\bar{B}$. Furthermore, this MGF is given as a sum of  the MGF of a Cox-Ingersoll-Ross process plus an extra term which is given by a nonlinear partial differential equation (PDE) on $\partial_t$ and $\partial_x$.  The  existence and the uniqueness of the solution of the nonlinear PDE is then proved using the inverse function theorem in a Banach space that will be defined in the paper. As an application, the mean reverting equation 
\begin{equation}\label{mean-reverting-abstract}
d V_t = ( a - b V_t) d t + \sigma V^p_t d W_t, ~~~V_0 = v_0 > 0,
\end{equation}
is extensively studied where some sharp asymptotic expansions of its MGF as well as its complementary cumulative distribution (CCDF) are derived.

 \medskip

\noindent

\medskip


\end{abstract}

\maketitle

\section{introduction}
\noindent In a previous article (\cite{Aly13}) we derived a formula which relates explicitly the moment generating function (MGF) and the complementary cumulative distribution function (CCDF)  of a random variable whose MGF is finite only on part of the real line. In situations  where MGF is available we obtain  the right tail behavior of the cumulative distribution function by only looking at the behavior of MGF near the critical moment. The main result in \cite{Aly13} is the following: we consider a random variable $Z$ whose moment generating function $ \mathbb{E}~ e^{ \mu Z} =: e^{\Lambda(\mu)}$ is known. If the function $\Lambda$ is finite in $[0,\mu^\ast[$ and explodes at $\mu^\ast$, for some $\mu^\ast>0$, and if the function $x \longmapsto \Lambda(\mu^\ast - \frac{1}{x})  $ behaves in a certain way (regularly varying with index $\alpha $ + some nonrestrictive conditions) as $x$ tends to $\infty$ then the Fenchel-Legendre transform of $\Lambda$ is  a good asymptotic approximation of the logarithm of  cumulative distribution of $X$ with $\limsup$-$\liminf$  arguments:
\begin{equation}\label{lim_sup}
\limsup_{x \to \infty} \frac{ \ln \mathbb{P} (X \geq x) + \Lambda^\ast (x)}{\ln (x)} \in \left[  - \frac{\alpha + 2}{2 (\alpha + 1)}, 0 \right]Ê Ê
\end{equation}
and 
\begin{equation}\label{lim_inf}
\liminf_{x \to \infty} \frac{ \ln \mathbb{P} (X \geq x) + \Lambda^\ast (x)}{\ln (x)} \le   - \frac{\alpha + 2}{2 (\alpha + 1)}, Ê
\end{equation}
where $\alpha$ is such that $x \longmapsto \Lambda(\mu^\ast - \frac{1}{x}) \in R_\alpha  $ (regularly varying with index $\alpha$). This result was motivated by several examples in the literature  where the moment generating function may be obtained explicitly. This is the case in several time-changed L\'evy models (see e.g. \cite{Aly13} and \cite{BenaimFriz08}). It has also been shown in \cite{Aly13} that any combination of the CIR process:
\[
d V_t = (a - bV_t) dt + \sigma \sqrt{V_t} d W_t,
\]
 and its integral satisfies this condition.

In this paper we consider a general class of one-dimensional stochastic differential equations:
\begin{equation}\label{main_eq}
d X_t = b(X_t) d t + \sigma(X_t) d W_t,~~~~X_0  > 0,
\end{equation}
 where we do not necessarily have an explicit expression for the MGF.  We first derive a new Tauberian result which is an improvement on (\ref{lim_sup}) and (\ref{lim_inf}). The new Theorem does not assume that the MGF is known or regularly varying. Instead it only needs to be bounded for below and above by two power functions. The result is an expansion of the complementary cumulative distribution in terms of the MGF: 
\begin{equation}\label{}
\mathbb{P} (X \geq x)~\sim~ e^{ - \Lambda^\ast (x)}  ~\frac{ \sqrt{ { p^\ast}' (x) } }{ p^\ast (x) \sqrt{2 \pi}}  ,
\end{equation}
where $\Lambda^\ast$ is the Fenchel-Legendre transform of $\Lambda: \mu \mapsto \ln \mathbb{E} e^{ \mu X}Ê$ and $ p^\ast (x) ={ \Lambda^\ast }' (x)$. 

The new result links the cumulative distribution and the MGF for a random variable that is bounded by two random variables whose MGFs explode at the same critical moment. This does not seem to help us obtain MGF nor CCDF, as we only link the "unknown" MGF to the "unknown" MGF, yet throughout this paper we shall see how this result will allow us to reduce the problem of obtaining the MGF (near its critical moment) for a sophisticated stochastic differential equation to solving a simple nonlinear partial differential equation. The main steps are the following: 
\begin{enumerate}[label =(\roman*)]
\item Using the comparison theorem, we "squeeze" our process between two processes (they will be two Cox-Ingersoll-Ross (CIR) processes), where the MGF is known;
\item We  find  (via Ito's formula) that the function $ \psi (t, \mu; \zeta): (t,\mu ; \zeta) \mapsto \mathbb{E} ( v^\zeta e^{ \mu V^\lambda_t})$ is a solution to a partial differential equation of the form $\partial_t \psi = \sum \nu_i \psi (t, \mu; \zeta + \chi_i)$ for some $ \nu_1, \dots, \chi_1, \chi_2 \dots$.  We then consider the function $\Delta(t,x) := \ln \psi (t, \mu^\ast - \frac{1}{x}; 0) $. The time derivative of $\Delta$ will depend on $ \frac{  \psi (t, \mu^\ast - \frac{1}{x}; \chi_i)}{ \psi (t, \mu^\ast - \frac{1}{x}; 0)}$. Will then come the role of the main Tauberian result of this paper which gives an expansion of $ \frac{  \psi (t, \mu^\ast - \frac{1}{x}; \chi_i)}{ \psi (t, \mu^\ast - \frac{1}{x}; 0)}$ for large $x$; the result is a nonlinear partial differential equation (PDE) for $\Delta$ only on $\partial_t$ and $\partial_x$. 
\item The existence and uniqueness of "the" solution of this PDE will be proven in a suitable Banach space using the inverse function theorem.
 \end{enumerate}

Our previous work in \cite{Aly13}  gave a pathway between the behavior to the moment generating function near its critical moment and the right tail of the cumulative distribution. The main result  of  \cite{Aly13} was an improvement on the results of Benaim and Friz in \cite{BenaimFriz08} who proved that if the MGF is regularly varying near $\mu^\ast$ then $ \ln \mathbb{P}( X >x) \sim - \mu^\ast x$. The proof of this statement relies on classical Tauberian theorems (see \cite{Bingham87} for more insight in Karamata's theory). And to our knowledge, this was the first result that allowed to going from the moment explosion (plus some condition on the behavior of the MGF near the critical moment) to $ \ln \mathbb{P}( X >x) \sim - \mu^\ast x$, in the vast theory of regular variation. We draw attention that regular variation theory is not the unique approach to derive asymptotic formulas for the distribution tails from the knowledge of the MGF near its critical moment; saddle point arguments allow in certain cases to derive some sharp asymptotics for the density, mainly using Mellin inversion formula; a good example is given in \cite{friz11}. Of course our context is different from the Saddle point arguments as we do not even have (yet) an explicit formula for the moment generating function and perhaps an even more important point is that there will be no mention of the density function in this paper. Indeed \textit{ neither the existence of a smooth density nor the uniqueness (strong or weak) of the solution of the SDE is required in this paper}. Some other works use different techniques, such as Large deviation and Malliavin calculus,  to treat some special types of SDEs. For example, in the case mean-reverting SDEs (with diffusion term given as a power function)  Conforti et all (\cite{Stefano15}) prove a pathwise large deviation principle for a rescaled process $ V^\epsilon =  \epsilon^\frac{1}{1-p} V $ when $p \in [\frac{1}{2}, 1[$. This allows deriving leading order asymptotics for path functionals of the process (namely $V_t$, $ \sup_{t \le T} V_t$ and  $\int_0^t V_s d s$). We study this type of SDE in this paper and derive a sharp asymptotic expansion for its MGF as well as its CCDF.

We end this section with a quick overview of the structure of the present paper. In Section~\ref{sec2}, after recalling  the main Tauberian result of \cite{Aly13}, we present a new Tauberian theorem.  In Section~\ref{sec3} we derive all the results related to a large class of one-dimensional SDEs; we first recall the results related to the Cox-Ingersoll-Ross process. We then study the SDE (\ref{SDE-nonlinear-Drift}) and we deduce the results related to a general one dimensional SDE via an easy transformation. We apply our results to (\ref{mean-reverting-abstract}) in Section~\ref{sec4}. The appendix contains the technical details and proofs of the results presented in Section~\ref{sec2} and Section~\ref{sec3}

 \section{Tauberian relation between the moment explosion and the distribution tails}\label{sec2}
 \noindent We first give a summary of the main Tauberian result derived in \cite{Aly13} followed by a  new Theorem which is an improvement on the main theorem of \cite{Aly13}.  
  
  Throughout this paper $(\omega, \mathcal{F}, (\mathcal{F}_t)_{t\geq 0}, \mathbb{P})$ is a complete filtered probability space satisfying the usual conditions and $\mathbb{E}$ refers to the expectation under $\mathbb{P}$. Consider a random variable $X$  whose moment generating function explodes at a critical moment $\mu^\ast < \infty.$ With an easy application of Markov inequality we can derive an upper bound for the cumulative distribution of $X$ as follows:
  \[
  \mathbb{P} (X>x)   \le e^{ - \sup_{0<\mu<\mu^\ast} ( \mu x - \Lambda (\mu))} =: e^{ - \Lambda^\ast (x)} 
~\Longrightarrow
\limsup_{x \to \infty} \frac{ \ln (\mathbb{P} (X>x)) + \Lambda^\ast (x)   }{\ln(x)} \le 0,
  \]
  where we used the same notation as \cite{Aly13}: $ \Lambda(\mu) := \ln (\mathbb{E} e^{ \mu X})$ and $ \Lambda^\ast$ is the Fenchel-Legendre of $ \Lambda$. The key idea in \cite{Aly13} is that instead of searching for a lower bound for the "$\liminf$",  we search for an upper bound for the "$\liminf$" and a lower bound for "$\limsup$".  Those bounds are obtained by using the following representation for the exponential function:
\[
  e^{p X} =  e^{ p( c \wedge X )} + p \int_c^\infty e^{ p z} 1_{ z \le X} d z .
\]
In particular, if we consider $p^\ast(x)$ defined for $x$ sufficiently large so that:
\[
\Lambda^\ast(x) = p^\ast(x) x - \Lambda(p^\ast(x)),
\]
  we have
  \begin{equation}\label{eq_mgf}
  e^{ \Lambda (p^\ast (x))} =  \mathbb{E} (   e^{ p^\ast (x) ( x^\beta \wedge X ) } ) + p^\ast (x) \int_{ x^\beta}^\infty e^{ p^\ast (x) z} \mathbb{P} ( X>z) d z .
  \end{equation}
  To derive a lower bound for the "$\limsup$" we prove that it can not be smaller than certain $\nu^\ast < 0$ to be determined under some assumption. We consider $ \nu' > \limsup_{x \to \infty} \frac{ \ln (\mathbb{P} (X>x)) + \Lambda^\ast (x)   }{\ln(x)} $; this means that for $z$ sufficiently large we have
  \[
  \mathbb{P} ( X> z ) \le e^{ \Lambda^\ast (z) + \nu' \ln (z) }.
  \]
  Plugging this into (\ref{eq_mgf}) and noticing that $ \mathbb{E} (  e^{ p^\ast (x) ( x^\beta \wedge X ) } ) \le  e^{ p^\ast (x) x^\beta} $ we have, choosing $ \beta$ such that $ x^\beta \ll \Lambda( p^\ast (x) )$,
  \[
  e^{ \Lambda (p^\ast (x))} \le 
  e^{ p^\ast (x) x^\beta} + p^\ast(x) e^{ \Lambda(p^\ast(x)) } x^{\nu'+1} \int_{ x^{\beta-1} }^\infty z^{ \nu'} ~ e^{  \chi_x (z)} d z,
 \]
 where  $\beta$ is chosen so that $x^\beta \ll \Lambda(p^\ast (x))  $  andÊ $\chi_x(.)$ is defined by 
 \begin{equation}\label{psi_x_z}
 \chi_x (z) = (p^\ast(x)- p^\ast(xz )) x z + \Lambda(p^\ast(xz ))- \Lambda(p^\ast(x)).
 \end{equation}
Here comes the role of the key (Tauberian) Lemma of \cite{Aly13} that states that, under suitable assumption on the function $p^\ast$,  we have,  for any $\nu \in \mathbb{R} $,
\begin{equation}\label{lemma_aly13}
\lim_{ x \to \infty}  \ln \left(  \int_{ x^{\beta-1} }^\infty z^{ \nu} ~ e^{  \chi_x (z)} d z \right)/\ln(x) = \frac{- \alpha}{2(\alpha + 1)}Ê.
\end{equation}
This is enough to show that $\nu'$ can not be smaller than $ \frac{-(\alpha + 2)}{2(\alpha + 1)} $.  Should this not be the case,  we would have for $x$ sufficiently large,
\[
e^{ \Lambda(p^\ast(x)} \le p^\ast(x) e^{ \Lambda(p^\ast(x)} ~Êx^{ -\zeta}, ~~~ \textrm{where}~~  \zeta := \nu' +  \frac{\alpha + 2}{2(\alpha + 1)}<0.
\]
This is obviously impossible, and it shows that $ \limsup_{x \to \infty} \frac{ \ln (\mathbb{P} (X>x)) + \Lambda^\ast (x)   }{\ln(x)}  \geq -  \frac{\alpha + 2}{2(\alpha + 1)} $. Using similar argument,  we prove that "$\liminf ()$" can not be larger than $ \frac{-(\alpha + 2)}{2(\alpha + 1)} $.

The results of \cite{Aly13} are derived under the assumption that $ x \mapsto \ln \mathbb{E} ~e^{ ( \mu^\ast - \frac{1}{x}Ê) X }$ is regularly varying, satisfying some differentiability conditions. In this article we shall consider the set ${\bf A_n (T, M; \gamma)}$ of {\it continuous functions on $ [0,T] \times [M, \infty[$ which are $n$-times continuously differentiable with respect to $x$, for some $ T, ~ÊM > 0$ and $n \in \mathbb{N}$, such that  for every $ k \le n$ , $ x^{ - \gamma + k}  \frac{ \partial f }{ \partial x^k} (t, x) $ is bounded on $ [0,T] \times [M, \infty[$}.  We shall denote $ {\bf A_n (M; \gamma)} $ the set of functions in $A_n (T,M; \gamma)$ which are {\it constant with respect to the variable $t$}.  We shall also denote $ {\bf \bar{A}_n (T, M; \gamma) }$ (resp. $\bar{A}_n (M; \gamma)$) the set of function $f$ such that $ f \in A_n (T, M; \gamma)$ and $ \frac{1}{f} \in   A_n (T, M; -\gamma)$ (resp.  $ f \in A_n (M; \gamma)$  and  $ \frac{1}{f} \in   A_n ( M; -\gamma$)).  The sets $ A_n$ share similar properties with the set of functions with regular variations and smooth variations used in \cite{Aly13}. In particular, it is well known that  if  $ f$ is regularly varying with index $ \alpha > 0$ in $ [M, + \infty[$ then $ f \in A_0 (M; \alpha + \epsilon) $, whenever $ \epsilon > 0$. %Also

 Next, we present the Tauberian results of this paper. The main result is a new version of (\ref{lim_sup}) and (\ref{lim_inf})  which applies to a random variable $X$ satisfying the following assumptions:

 \begin{assumption}\label{hyp_laregdev}
There exist $ \mu^\ast > 0$,  $ \alpha_1 $ and $ \alpha_2$, with $ 0< \alpha_1 \le \alpha_2$ and $ \alpha_2 ( 1 + \frac{1}{1 + \alpha_2} ) < 2 \alpha_1$, and $ c_1, c_2 ~Ê\in \mathbb{R}_+$ such that for every $ 0 \le \mu < \mu^\ast$, 
\begin{equation}
\frac{c_1}{ (\mu^\ast - \mu)^{\alpha_1} } \le  \ln \mathbb{E} \; e^{\mu X} \le \frac{c_2}{ (\mu^\ast - \mu)^{\alpha_2 }} 
\end{equation}
\end{assumption}

The next Lemma will be crucial for the proof of the main Tauberian Lemma~\ref{lemma-tauberian} and  Theorem~\ref{theorem-tauberian}.

\begin{lemma}\label{lemma-bounds}
Under Assumption~\ref{hyp_laregdev} we have $ \Lambda : p \mapsto \ln ~\mathbb{E}~e^{ p X} 1_{ X > 0}$ is convex. In particular $ \Lambda^\ast (x) = p^\ast (x) x - \Lambda (p^\ast(x))$, where $p^\ast (x)$ is the unique solution to $ x = \Lambda' (p^\ast (x))$. Furthermore, 
\begin{equation}
\mu^\ast - \tilde{c_2}x^\frac{-1}{\alpha_2 + 1} \le p^\ast (x) \le \mu^\ast - (\frac{c_1}{\tilde{c}_2}) x^\frac{-\alpha_2/\alpha_1}{\alpha_2 + 1}.
\end{equation}
and
 \begin{equation}
 {p^\ast}' (x) \in \left[  
 m_1 x^{ - \frac{\alpha_2}{\alpha_1}( 1 +  \frac{1 }{1 + \alpha2} )Ê},
 m_2 x^{ - \frac{\alpha_1}{\alpha_2} (1 +  \frac{1 }{1 + \alpha2} )Ê},
 \right],
 \end{equation}
 where $ \tilde{c}_2 = ( c_2 \alpha_2)^\frac{1}{1 + \alpha_2} (1 + \alpha_2)/\alpha_2$, and $ m_1, m_2 > 0$.
\end{lemma}

\begin{lemma}\label{lemma-tauberian}
Let Assumption~\ref{hyp_laregdev} hold for some random variable $X$. Consider the function  $ \chi_x(.)$ defined by (\ref{psi_x_z}), for $x$ sufficiently large and let $g \in A_0 (0; \gamma)$, for some $ \gamma \in \mathbb{R}$. Then we have
\begin{equation}\label{taub_res}
 \int_{ x^{\beta-1} }^\infty g (xz) ~ e^{  \psi_x (z)} d z \sim_{ x \to \infty} \frac{ g(x) \sqrt{ 2 \pi}}{ \sqrt{ 2  x^2  {p^\ast}'(x)}   } ,
\end{equation}
where $ \sim$ refers to the equivalence between two functions defined as $ f_1 \sim f_2$  if $~~\lim_{x \to \infty} \frac{f_1 (x)}{f_2  (x) } = 1 $. Furthermore, if $x \mapsto \Lambda (\mu^\ast - \frac{1}{x}Ê)\in \bar{A}_3 (M; \alpha)$, for some positive $ \alpha$  then for every continuous function $g$ such that $ \lim_{x \to \infty} \frac{x g'(x)}{g(x)} = \gamma  $, we have
\begin{equation}\label{taub_res-refined}
 \int_{ x^{\beta-1} }^\infty g (xz) ~ e^{  \psi_x (z)} d z = \frac{ g(x) \sqrt{ \pi}}{ \sqrt{ 2  x^2  {p^\ast}'(x)}   } \left( 2 + \frac{ \gamma^2 + \frac{\gamma}{\alpha + 1}  + c_\alpha  }{x^2  {p^\ast}'(x) }  +  o ( \frac{1}{x^2 {p^\ast}'(x) }) \right),
\end{equation} 
where
 \begin{eqnarray*}
c_\alpha &=&  - \frac{1}{4} (1 +  \frac{1}{\alpha + 1}) (2 +  \frac{1}{\alpha + 1}) + \frac{5}{12} ( 1 +  \frac{1}{\alpha + 1})^2.
\end{eqnarray*}
\end{lemma}

\begin{theorem}\label{theorem-tauberian}
Under Assumption~\ref{hyp_laregdev}  we have
\begin{equation}\label{lim_sup2}
\limsup_{x \to \infty} ÊÊ\left(   \ln \mathbb{P} (X \geq x) + \Lambda^\ast (x)  +\ln ( \frac{ p^\ast(x) \sqrt{2 \pi} }{  \sqrt{ {p^\ast}' (x) }  } ) \right) \geq 0, 
\end{equation}
and 
\begin{equation}\label{lim_inf2}
\liminf_{x \to \infty} ~ÊÊ\left(   \ln \mathbb{P} (X \geq x) + \Lambda^\ast (x)  +\ln ( \frac{ p^\ast(x) \sqrt{2 \pi} }{  \sqrt{ {p^\ast}' (x) }  } ) \right) \le 0. Ê
\end{equation}
In particular, if   "$\limsup$" equals "$\liminf$"  then we have 
\begin{equation}\label{expans_cdf}
\mathbb{P} (X \geq x) ~ \sim ~ e^{ - \Lambda^\ast (x)}   \frac{ \sqrt{ { p^\ast}' (x) } }{ p^\ast (x) \sqrt{2 \pi}} .
\end{equation}
Furthermore, if $x \mapsto \Lambda (\mu^\ast - \frac{1}{x}Ê) \in \bar{A}_3 (M; \alpha)$, then we have (when "$\limsup$" equals "$\liminf$")
\begin{equation}\label{expans_cdf-2}
\mathbb{P} (X \geq x) = e^{ - \Lambda^\ast (x)}  \left( \frac{ \sqrt{ { p^\ast}' (x) } }{ p^\ast (x) \sqrt{2 \pi}} - \frac{ 2 + \frac{\alpha}{ (\alpha + 1)^2 } }{  24 \mu^\ast  \sqrt{2 \pi} } ~\frac{1}{x^2 \sqrt{ {p^\ast}'(x) }  }    + o(     \frac{1}{x^2  \sqrt{ {p^\ast}' (x) } }    )   \right).
\end{equation}
\end{theorem}

\begin{proposition}\label{proposition-tauberian}
Let $X$ and $g$ satisfy the assumptions of Theorem~\ref{theorem-tauberian} and assume that "$\limsup$" and "$\liminf$" in (\ref{lim_inf2}) and (\ref{lim_sup2}) are equal.  ÊWe have, 
\begin{equation}
\frac{ \mathbb{E} ( g(X) e^{( \mu^\ast - \frac{1}{x} ) X} )}{ \mathbb{E} (  e^{( \mu^\ast - \frac{1}{x} ) X} ) } ~ \sim_{ x \to \infty}~ g( \Lambda' (\mu^\ast - \frac{1}{x}) ) + \frac{1}{\mu^\ast - \frac{1}{x} } g' ( \Lambda' (\mu^\ast - \frac{1}{x}) )Ê .   
\end{equation}
If in addition   $x \mapsto \Lambda (\mu^\ast - \frac{1}{x}Ê)\in \bar{A}_3 (M; \alpha)$, then we have  for  any $\mathcal{C}^1-$function $g$ such that $ \lim_{ x \to \infty } \frac{x g' (x) }{ g (x)} = \gamma \in \mathbb{R} $,
\begin{equation}\label{tauberian-pososition-refined}
\frac{ \mathbb{E} ( g(X) e^{( \mu^\ast - \frac{1}{x} ) X} )}{ \mathbb{E} (  e^{( \mu^\ast - \frac{1}{x} ) X} ) } ~ \sim_{ x \to \infty}~ g( \Lambda' (\mu^\ast - \frac{1}{x}) )   \left( 1 + (\gamma^2 - \gamma) \frac{ \Lambda'' (\mu^\ast - \frac{1}{x}  )  }{  2 { \Lambda'}^2 ( \mu^\ast - \frac{1}{x}) } + o( \frac{ \Lambda'' (\mu^\ast - \frac{1}{x}  )  }{  2 { \Lambda'}^2 ( \mu^\ast - \frac{1}{x}) } Ê )   \right) .
\end{equation}
\end{proposition}

\section{One-dimensional stochastic differential equations}\label{sec3}
\noindent The Tauberian result obtained in \cite{Aly13} applies to any random variable whose moment generating function is known and regularly varying near its critical moment. Theorem~\ref{theorem-tauberian} relaxes the regular varying assumption to being bounded from below and above by two power functions. This obviously is a large improvement on the results of \cite{Aly13}, but the main advantage of Theorem~\ref{theorem-tauberian} is that we do not need to have the explicit expression of the MGF to link it the cumulative distribution; instead we only need the knowledge  of the two bounds. In what follows we  consider a stochastic differential equation and use the comparison theorem to "squeeze" it between two SDEs whose MGF are known. Then by applying Theorem~\ref{theorem-tauberian} and Proposition~\ref{proposition-tauberian} we reduce the problem of obtaining the tail of  distribution to solving a nonlinear partial differential equation. We then prove the existence and the uniqueness of the nonlinear PDE using the inverse function theorem in a Banach space that we shall define later in this article.

 We consider a general continuous stochastic differential equation of dimension one 
\begin{equation}
d X_t = B (X_t) d t + \sigma(X_t) d W_t.
\end{equation} 
We will focus on stochastic differential equations which admit a positive solution (see Remark~\ref{positive-negative-SDE} below for the discussion about the general case of general stochastic differential equation). We start by considering the easy case where the MGF is known; we consider the Cox-Ingersoll-Ross process (CIR). In the following step we consider the case of $ \sigma(x) = \sqrt{x}$ and a general form for $B(.)$. In this step we shall use the CIR processes as bounds for our SDE. In the third step, we consider the general setting for $ \sigma(.)$ and $B(.)$.

\subsection{Cox-Ingersoll-Ross (CIR) process} The case of linear drift term:  $ b (x) = a - bx$ and  square-root diffusion: $ \sigma(x) = \sigma \sqrt{x}$, where $ b \in \mathbb{R}$ and $ a, \sigma> 0$ has been studied extensively in the literature. Its moment generating function $ \mathbb{E} e^{ \mu X_t}$ is known (see e.g. \cite{Aly13}, appendix~A). In particular for $ \mu > 0 $ such that $ | \mu - \frac{b}{ \sigma^2} | > \frac{|b|}{\sigma^2}$ (i.e. $ \mu > \max(0, \frac{2b}{\sigma^2} ) $) we  have
\[
\ln \mathbb{E} e^{ \mu X_t }=   a \varphi (t) + X_0 \psi(t),
\]
where $ \varphi (t) = \int_0^t \psi (s) d s$ and $ \psi$ is given by
\[
\psi(t) = \frac{b}{\sigma^2} -  \frac{|b|}{\sigma^2} \frac{ C(\mu)  e^{ |b| t  }  +1}{C (\mu) e^{ |b| t  }  -1 } Ê, 
\]
with 
\[
C (\mu)  :=  \frac{ \mu \sigma^2 - b - |b| }{Ê\mu \sigma^2 - b +|b|} = ( \frac{ \mu \sigma^2 - 2 b }{ \mu \sigma^2}Ê)^{ \sign (b)}.
\] 
We see clearly that the MGF explodes at the critical moment $ \mu^\ast_t$ satisfying 
\[
C (\mu^\ast) e^{ |b| t  }  -1 = 0.
\]
We easily find that:
\[
\mu^\ast_t = \frac{2b}{ \sigma^2 ( 1 - e^{ - b t})} .
\]
Furthermore, we have, for $ \mu < [  0 \vee \frac{2 b}{ \sigma^2Ê},  \mu^\ast_t [$,  
\[
\frac{ C(\mu)  e^{ |b| t  }  +1}{C (\mu) e^{ |b| t  }  -1 } = \frac{ C(\mu)  + C(\mu^\ast_t)}{C(\mu)  - C(\mu^\ast_t) }  = -\frac{\sigma^2}{2 |b|}  \frac{ 2 \mu \mu^\ast_t  - \frac{2 b}{\sigma^2}  (\mu^\ast_t + \mu)}{ \mu^\ast_t - \mu}  =-\frac{\sigma^2}{ |b|}  \frac{  \mu \mu^\ast_t  e^{- b t}}{ \mu^\ast_t - \mu} +\sign(b) .
\]
We find, after some straightforward computations, that for $ \mu < [  0 \vee \frac{2 b}{ \sigma^2Ê},  \mu^\ast_t [$ 
\begin{eqnarray}\label{mgf_cir}
  \ln \mathbb{E} ~e^{\muÊX_t} &=& \frac{X_0 e^{ - b t } {  \mu^\ast}^2_t}{\mu^\ast - \mu}   + 2a   \ln ( \frac{{ \mu^\ast_t} }{\mu^\ast_t - \mu} ) - X_0  e^{ - b  t} \mu^\ast_t .
\end{eqnarray}
 This MGF satisfies Assumption~\ref{hyp_laregdev} and we may, therefore, use Theorem~\ref{theorem-tauberian}  to derive a sharp expansion of its cumulative distribution.

\subsection{Square root equation with nonlinear drift}
Let's now consider the stochastic differential equation:
\begin{equation}\label{NonliearSR}
d X_t = B(X_t) d t + \sigma  \sqrt{X_t} d W_t, ~~~X_0 >0.
\end{equation}
%where   $ b, ~ \eta \in \mathbb{R}$.  
We shall assume that $B$ satisfies the following assumptions

\begin{assumption}\label{assum_1}
The function $ B$ is bounded at $0$ and satisfies the following assumptions:
\begin{enumerate}[label=(\roman*)] 
\item The function $ y \mapsto \frac{B(y)}{y} $ is monotonic on $[M, +\infty[$ for some $M>0$, such that there exists $ b  >0$  such that   Ê
\begin{equation}
\lim_{x \to \infty} \frac{B(x)}{x} = -b .
\end{equation}
\item The function $ \bar{B} (y)   := B(y) + b y \in \bar{A}_0 (M; \beta) $ for some $ \beta < 1$.
\item If $ y \mapsto \frac{B(y)}{y} $ is monotonically increasing on $[M, +\infty[$ then $\beta \le \frac{1}{2} $.
\end{enumerate}
\end{assumption}

By comparison theorem arguments we can "squeeze" the process $X$ between two Cox-Ingersoll-Ross (CIR) processes which have the same critical moment. This will immediately imply that the MGF of $X_t$ explodes at the same critical moment. On the other hand, using the explicit formulas for the MGF of CIR process, we find that $X$ satisfies Assumption~\ref{hyp_laregdev}. This is given in the next result

\begin{theorem}\label{moment-explosion}
Let Assumption~\ref{assum_1}  
 hold and denote by 
\begin{equation}\label{critical-moment}
\mu^\ast_t := \frac{2 b}{ \sigma^2 (1 - e^{ - b t} )} .
\end{equation}
Then for every $ t > 0$, there exist  $ \omega_1 (t)$ and $ \omega_2 (t)$ such that for any  $ \mu < \mu^\ast_t$,
\begin{equation}
\frac{\omega_1(t)}{ \mu^\ast - \mu } \le \ln~\mathbb{E} ~e^{\mu X_t} \le   ~   \frac{  \omega_2 (t) | \log(\mu^\ast - \mu )^\frac{1}{1-\beta} |      }{ (\mu^\ast - \mu)^{ \frac{\beta}{\beta-1} \vee 1} }.
\end{equation}
In particular, the results of Theorem~\ref{theorem-tauberian} and Proposition~\ref{proposition-tauberian} hold for $X_t$.
\end{theorem}

\begin{proof}
We outline the main steps of the proof. The technical details are given in appendix~\ref{poof-ME}. Using the monotonicity of $ y \mapsto \frac{B(y)}{ y}Ê$ on $ [M, \infty[$ we prove that there exist $c$ and $ m$ such that  for $ Z$ sufficiently large and  $ y \geq 0$ we have $ c  Z^\beta + \frac{B(Z)}{ Z}Êy  \geq  ~B(y) \geq ~-   m - by$ if $\frac{B(y)}{ y}$ is decreasing and $ -c Z^\beta+ \frac{B(Z)}{ Z}Êy   \le ~B(y)  \le~ m - by$ if $\frac{B(y)}{ y}$ is increasing. This means that using the comparison theorem, the process $ X$ is bounded from below and above by two CIR processes $ V^{ \pm c Z^\beta , - \frac{B(Z)}{  Z} , X_0 }$ and  $ V^{  \mp m ,  b, y_o}$, where $ V^{a,\kappa, X_0}$ is the given by the SDE
\[
d V^{a,\kappa, X_0}_t = ( a - \kappa V^{a,\kappa, X_0}_t ) d t + \sqrt{V^{a,\kappa, X_0}_t} d W_t, ~~~~Ê  V_0 = X_0.
\] 
Let $x$ be sufficiently large and $ Z \equiv Z(x) = (x \log(x))^\frac{1}{1-\beta} $. We find using the explicit expression of the MGF of  $ V^{ \pm c Z^\beta , - \frac{B(Z)}{  Z} , X_0 }$ and  $ V^{  \mp m ,  b, x_o}$ that 
 \begin{enumerate}[label=(\roman*)]
\item If $\frac{B(y)}{ y}$ is decreasing then there exist $ c_1 (t), ~c_2 (t),~ c_3 (t) > 0 $ such that
\[
c_1 (t) x \le \log Ê\mathbb{E} e^{ ( \mu^\ast_t - \frac{1}{x} ) X_t} \le c_2 (t) x + c_3 (t) \ln (x)|^\frac{1}{1-\beta} x^\frac{\beta}{1-\beta} .
\]
\item  If $\frac{B(y)}{ y}$ is increasing then there exist $ d_1 (t), ~d_2 (t),~ c_3 (t) > 0 $ such that
\[
d_2 (t) x - d_3 (t) \ln (x)|^\frac{1}{1-\beta}  x^\frac{\beta}{1-\beta} \le \log Ê\mathbb{E} e^{ ( \mu^\ast_t - \frac{1}{x} ) X_t} \le d_1 (t) x.
\]
In particular if $ \beta < \frac{1}{2}$, then $  d_3 (t) \ln (x)|^\frac{1}{1-\beta}  x^\frac{\beta}{1-\beta} \ll d_2 (t) x$ for $x$ sufficiently large. Hence,
\[
\omega_1 (t) x \le d_2 (t) x - d_3 (t) \ln (x)|^\frac{1}{1-\beta}  x^\frac{\beta}{1-\beta} \le \log Ê\mathbb{E} e^{ ( \mu^\ast_t - \frac{1}{x} ) X_t} \le d_1 (t) x.
\]
\end{enumerate}
This means that if we denote $ \gamma := \frac{\beta}{1-\beta}\vee 1Ê$, then $X_t$ satisfies Assumption~\ref{hyp_laregdev} whenever $  2 \gamma ( 1 + \frac{1}{1 +  2 \gamma }Ê) < 2$ (which holds when $\gamma <\frac{3}{2} $). Hence Theorem~\ref{theorem-tauberian} and Proposition~\ref{proposition-tauberian} apply to $X_t$ when $ \gamma \le \frac{3}{4}  $ (i.e. $\beta \le \frac{3}{5}$).  To apply those results to the case where $ y \mapsto \frac{B(y)}{ y}Ê$ is decreasing such that $ |\bar{B} (y)| \geq y^\frac{3}{5} $ we use  equation $ d X_t = ( - bX_t  +  X^\frac{3}{7}_t ) d t + \sqrt{X_t} d W_t $   as a lower bound. We shall see that for this process, there exists $ c(t) > 0$ such that $   \log ~\mathbb{E} ~e^{ ( \mu^\ast - \frac{1}{x}Ê) X_t} \sim  c(t) x^{  1 \vee \frac{3/5}{1 - 3/5}Ê} = c(t)x^\frac{3}{2}  $. Hence Assumption~\ref{hyp_laregdev} is satisfied whenever $ \frac{\beta}{1 - \beta}   ( 1 + \frac{1}{1 + 2 \frac{\beta}{1 - \beta} }Ê) < 2 \times \frac{3}{2}Ê$, which holds whenever $ \beta \le \frac{9}{14}$. Repeating this procedure, we find that the results heorem~\ref{theorem-tauberian} and Proposition~\ref{proposition-tauberian} apply to $X_t$ whenever Assumption~\ref{assum_1} is satisfied.
\end{proof}

Theorem~\ref{moment-explosion} does not give the explicit expression of the MGF but give the two bounds needed to apply the Tauberian  Theorem~\ref{theorem-tauberian} which links the MGF to the cumulative distribution function. The application of Theorem~\ref{theorem-tauberian} will give next result which gives the MGF as a combination of two terms: the first comes from the CIR part of SDE (and hence explicitly known) while the second term is given by a nonlinear partial differential equation. The rest of the paper will be devoted to the study of the second term. 

\begin{proposition}\label{MFG-to-CD}
Let's define, for $ t>0$ and $x$ sufficiently large,
\begin{equation}\label{Gamma-definition}
\Gamma(t,x) := \ln \mathbb{E} e^{ (\mu^\ast_t - \frac{1}{\xi_t  (x+ \mu^\ast_t  )} ) X_t }  , ~~~~\textrm{where} ~~\xi_t = e^{ b t} { \mu^\ast_t}^{ -2}. 
\end{equation}
We have
\begin{equation}\label{eq-lambda}
\Gamma (t,x) =  (\frac{2b}{\sigma^2 } + x) X_0 + \int_0^t (\mu^\ast_s - \frac{1}{ \xi_s  (x+ \mu^\ast_s  )   }) \tilde{B} ( s,\xi_s  (x+ \mu^\ast_s  )^2 \partial_x \Gamma (s,x) ) d s,
\end{equation}
where
\begin{equation}\label{def-tildeB}
\tilde{B} (t,x) := \frac{  \mathbb{E}  ( \bar{B} (X_t) e^{ (\mu^\ast_t -  \frac{1}{ \xi_t  (x+ \mu^\ast_t  )}) X_t }      )   }{ \mathbb{E}   e^{ (\mu^\ast_t -  \frac{1}{ \xi_t  (x+ \mu^\ast_t  )}) X_t }  } .
\end{equation}
\end{proposition}

\begin{proof}
Denote, for a measurable function $ f  $ such that $ f(Y_t)$ is integrable, $ \psi^f (t, \mu) := \mathbb{E} ( f(X_t) e^{ \mu X_t})$. From It\^o's formula we have
\[
\partial_t \psi^1 (t, \mu) = \mu \psi^{B} + \frac{1}{2} \mu^2 \partial_\mu \psi^1,
\]
where $ \psi^1 (t, \mu) := \mathbb{E} (   e^{ \mu Y_t})$. Let's also denote $\Delta (t, x) := \ln \psi^1 (t, \mu^\ast_t - \frac{1}{x})$ and consider the function $ \Gamma$ defined in (\ref{Gamma-definition}). After some straightforward computations (mainly using the fact that $ \partial_t \mu^\ast = - \frac{1}{2 \xi_t} $ and that $\partial_t \xi_t = -(b - \mu^\ast) \xi_t$), we find that  $ \Gamma$ is given as a solution to the partial differential equation
\[
\partial_t \Gamma (t,x) = \frac{ ( x + 2b ) \mu^\ast_t}{ x + \mu^\ast} \frac{   \psi^{\bar{B}}  }{ \psi^1} (t, \mu^\ast_t - \frac{1}{ \xi_t  (x+ \mu^\ast_t  )}) .
\]
We emphasize that 
\[
\lim_{t\to 0} \Gamma (t,x) = \lim_{t\to 0} ~\ln \mathbb{E} e^{ (\mu^\ast_t - \frac{1}{ \xi_t  (x+ \mu^\ast_t  )   }) Y_t} =  (2b + x) X_0,
\]
and
\[
\lim_{t\to 0} \partial_x \Gamma (t,x) =  \lim_{t\to 0} ~ \partial_x \ln \mathbb{E} e^{ (\mu^\ast_t - \frac{1}{ \xi_t  (x+ \mu^\ast_t  )   }) Y_t} = \frac{  \mathbb{E} ( \frac{1}{\xi_t  (x+ \mu^\ast_t  )^2}Êe^{ (\mu^\ast_t - \frac{1}{ \xi_t  (x+ \mu^\ast_t  )   }) Y_t}  )   }{    \mathbb{E} e^{ (\mu^\ast_t - \frac{1}{ \xi_t  (x+ \mu^\ast_t  )   }) Y_t} }  = X_0.
\]
Both $\mu^\ast_t - \frac{1}{ \xi_t  (x+ \mu^\ast_t  )   } $ and $ \xi_t  (x+ \mu^\ast_t  )^2 $ are bounded as $t \to 0$. 
The function $ \Gamma$ is then given as a solution the nonlinear differential equation:
\[
\Gamma (t,x) =  (\frac{2b}{\sigma^2} + x) X_0 + \int_0^t (\mu^\ast_s - \frac{1}{ \xi_s  (x+ \mu^\ast_s  )   }) \tilde{B} ( s , \xi_s  (x+ \mu^\ast_s  )^2 \partial_x \Gamma (s,x) ) d s.
\]
\end{proof}

The next few results discuss the existence and uniqueness of the solution of equation  (\ref{eq-lambda}). Let's first define $R(t,x): = \Gamma (t,x) -  (2b + x) X_0 $. Equation (\ref{eq-lambda}) is then equivalent to the following equation for $R:$
 \begin{equation}\label{R_eq}
 R(t,x) = \int_0^t  \frac{(x + 2 b ) \mu^\ast_sÊ}{ x+ \mu^\ast_s     } \tilde{B} \left[ \xi_s  (x+ \mu^\ast_s  )^2  (X_0 + \partial_x  R (s,x)   ) \right] d s.
 \end{equation}
The first result gives the asymptotic behavior of $ \tilde{B}$. Then we define a Banach space under which Equation (\ref{R_eq}) has a unique solution.

\begin{lemma}\label{asymptotic-Btilda}
Let Assumption~\ref{assum_1} hold and assume that $ \bar{B} \in A_\infty (0; \beta)$. Then $ \tilde{B} \in A_\infty (T, 0; \beta)$
\end{lemma}

\begin{proof}
By Proposition~\ref{proposition-tauberian} we have
\[
\frac{ \mathbb{E} ( \bar{B}(X) e^{( \mu^\ast - \frac{1}{x} ) X} )}{ \mathbb{E} (  e^{( \mu^\ast - \frac{1}{x} ) X} ) } \sim \bar{B}( x^2 \partial_x \Delta (t, x) )    + \frac{1}{\mu^\ast_t - \frac{1}{x} } \bar{B}' ( x^2 \partial_x \Delta (t, x) ) \sim  \bar{B}( x^2 \partial_x \Delta (t, x) ),
\]
where in Proposition~\ref{proposition-tauberian}'s statement, $ \Lambda' ( \mu^\ast - \frac{1}{x}Ê)$ refers to $ \partial_\mu \ln \mathbb{ E} e^{  (\mu^\ast - \frac{1}{x}Ê) Y_t }$  and is therefore given  by $ x^2 \partial_x \Delta (t , x) $.  Hence, $  \tilde{B} \in A_0 (T, 0; \beta)$.  We see clearly that the behavior of the derivatives of $ \tilde{B}  $  will follow those of $ \bar{B}  $ via Proposition~\ref{proposition-tauberian} and the differentiation inside the expectation. In particular the fact that $  \bar{B} \in A_n ( 0; \beta)$ will imply that $  \tilde{B} \in A_n (T, 0; \beta)$. Thus $  \tilde{B}_\infty \in A_0 (T, 0; \beta)$.
\end{proof}

\begin{lemma}\label{banach-lemma} 
Let $ T, M > 0$ and $ \gamma \geq 0$. The set $ A_\infty (T, M; \gamma)$  is a Banach space with the norm 
\begin{equation}\label{normX}
\|ÊÊu\| =  \sup_{ k \geq 0} \sup_{ (t,x) \in  [0,T] \times [M, \infty[ } x^{-\gamma+ k} | \frac{ \partial^k u}{\partial x^k} (t,x) |.
\end{equation}
\end{lemma}

\begin{theorem}\label{uniqueness-theorem}
 Let Assumption~\ref{assum_1} hold and assume that $\bar{B} \in A_\infty (T, 0; \beta)$. 
 Then for any $ T >0$, and $\gamma >1 \vee \frac{\beta}{1-\beta}$, there exists $ M > 0$ such that equation
 \begin{equation}\label{equation-rest}
 u(t,x) = \int_0^t  \frac{ (x + 2 b ) \mu^\ast_sÊ}{ x+ \mu^\ast_s     } \tilde{B} \left[s, \xi_s  (x+  \mu^\ast_s   )^2  (X_0 + \frac{\partial u}{\partial x} (s,x)  ) \right] d s
 \end{equation}
 admits only one solution in $A_\infty (T, M, \gamma)$. In particular  $ \Gamma (x) = 2 b X_0 + X_0 x + R(t, x)$, where $R (t, x) = \lim_{ n\to \infty} R_n (t, x)$ , with the sequence $ R_n (t, x)$ is defined by
 \[
 R_0 (t, x) = 0, ~~~ R_{n+1} (t, x) := \int_0^t  \frac{ (x + 2 b ) \mu^\ast_sÊ}{ x+ \mu^\ast_s     } \tilde{B} \left[s, \xi_s  (x+  \mu^\ast_s   )^2  (X_0 + \frac{\partial R_n}{\partial x} (s,x)  ) \right] d s.
 \]
 Furthermore,  if $ \bar{B} (y) \sim   c y^\beta$ then the function $ \tilde{R} ( \tau, y):= \frac{1}{X_0} R ( \frac{\ln(\tau)}{b \beta},  y - \frac{2b}{\sigma^2 }    )   $  satisfies $ \tilde{R} (t, y) \sim c(\tau) y^{ 2 \beta \vee \frac{\beta}{1 - \beta}Ê}$, where $c$ is given as
\begin{equation}
c_t = \omega_t 1_{ \beta < \frac{1}{2}   } +    (( 1 + \frac{1}{2}  \omega_t  )^2 - 1) 1_{\beta = \frac{1}{2}Ê}  +  ( \gamma^\beta (1 - \beta) )^{ \frac{1}{1-\beta}  Ê}  \omega^\frac{1}{1 - \beta}  1_{ \beta > \frac{1}{2}Ê},
\end{equation}
with 
\begin{equation}
\omega_\tau = c \frac{ x^{2 \beta}_0 }{ b \beta }Ê (\frac{2 b}{\sigma^2}Ê)^{ 1 - 2 \beta} \int_1^\tau  ( 1 - s^\frac{-1}{\beta}Ê)^{ 2 \beta - 1} d s
\end{equation} 
 \end{theorem}

\begin{remark}
If $ \bar{B} (y) =  c y^\beta$ then using Theorem~\ref{uniqueness-theorem} we have $ \Gamma \in \bar{A}_\infty (T, M; 1 \vee \frac{\beta}{1 - \beta}Ê)$. Hence we may apply the refined statements of Theorem~\ref{theorem-tauberian} and Proposition~\ref{proposition-tauberian} (i.e. (\ref{expans_cdf-2}) and (\ref{tauberian-pososition-refined})) to get a more refined expansion for $ \tilde{B}$. Solving equation (\ref{R_eq}) with the refined version of $ \tilde{B}$ will give a more refined expansion for $R$ and hence a more refined expansion for the cumulative distribution of $X_t$ using Theorem~\ref{theorem-tauberian}. For more details see next section where we study the mean-reverting case in details.
\end{remark}

\subsection{One dimensional stochastic differential equation}
Consider the stochastic differential equation
\[
d X_t = b(X_t) d t + \sigma(X_t) d W_t, ~~~ X_0 > 0,
\]
and define the process
\[
Y_t = \Sigma(X_t), ~~~Ê\textrm{where}~~~ \Sigma(x) :=\left(   \int_0^x \frac{d u}{ 2 \sigma (u) } \right)^2.
\]
Applying Ito's formula to $Y$ we have
\[
d Y_t = \left(   \frac{1}{4} - \frac{\sigma' (X_t)}{2} \sqrt{Y_t} + \frac{ b(X_t) }{ \sigma(X_t) }  \sqrt{Y_t}     \right) dt + \sqrt{Y_t}  d W_t
\]
It's worth emphasizing that from the definition of $ \Sigma$ we have
\[
\sigma (x) = \frac{\sqrt{ \Sigma (x) }}{ \Sigma'(x)}. 
\]
It follows that
\[
 \sigma( \Sigma^{-1} (y) ) = { \Sigma^{-1} }' (y) \sqrt{y},
\]
and
\[
 \sigma' ( \Sigma^{-1} (y) ) =    \frac{  ( \sigma( \Sigma^{-1} (.) ) )'(y)}{ { \Sigma^{-1} }' (y)} = \frac{1}{2 \sqrt{y}} + \frac{ { \Sigma^{-1} }'' (y)   }{{ \Sigma^{-1} }' (y)} Ê \sqrt{y}.
\]
Hence,
\[
d Y_t = B(Y_t) + \sqrt{Y_t} d W_t,
\]
where $B$ is given by
 \[
B:  y \mapsto - \frac{ { \Sigma^{-1} }'' (y)   }{{ \Sigma^{-1} }' (y)} y  +  \frac{ b ( \Sigma^{-1} ( y) ) }{  { \Sigma^{-1} }' (y)}   .
 \]
 We may then apply the result obtained in the previous subsection if $B$ satisfies  Assumption~\ref{assum_1}. It is worth emphasizing that if $ \Sigma \in \bar{A}_2 (M; \lambda)$, then 
\[
 \frac{{ \Sigma^{-1} }'' (y)   }{{ \Sigma^{-1} }' (y)} y \sim ( \lambda - 1) ~~~\textrm{and}~~~~\lim_{y \to \infty}  \frac{ b ( \Sigma^{-1} ( y) ) }{  { \Sigma^{-1} }' (y)    }  \frac{1}{y} =  Ê  \frac{1}{\lambda}  \lim_{x \to \infty} \frac{b (x) }{x}
\]

\begin{remark}[Positive and negative SDE]\label{positive-negative-SDE}
All the results in this paper apply to nonnegative processes given by continuous stochastic differential equations. To apply the results to a more general class of SDEs we proceed as follows:   For $ n > 0$ define $ X^+_n (t) = f_n (X_t) $, where
  \[
  f_n (x) = x ~Êe^{ \frac{-1}{n x} } 1_{x > 0} .
  \]
In particular, $ X^+_n (t) \to X_t 1_{X_t > 0}$. It is also worth emphasizing that $f_n$ is increasing in the positive line, with $ f(0) = 0$. In particular for every $y> 0$ there exists a unique $ x = f^{-1}_n (y)$ such that $ f_n (x) = y$. Now applying Ito's formula to $X_n$ we have
\[
d X^+_n (t) = ( 1 + \frac{1}{n X_t}  )e^{ \frac{-1}{n X_t} }1_{X_t > 0} ( b(X_t) + \sigma(X_t) d W_t) + \frac{1}{2 n^2 X^2_t}  e^{ \frac{-1}{n X_t} } 1_{X_t > 0} \sigma^2 (X_t) d t,
\]
where $ X_t = f^{-1}_n (X_n(t))$.
\end{remark}

\section{Application: Mean-reverting process}\label{sec4} 
\noindent  Consider a diffusion  $(V_t)_{t \geq 0}$ given by the stochastic differential equation
\begin{equation}\label{V_p}
d V_t  =  (a - b V_t ) d t + \sigma  V^p_t d W_t ,Ê~~~V_0 = v_0,
\end{equation}
where $W$ is a Brownian motion, $b,~a,~ v_0, ~ \sigma ~>0$ and $p \in ]0,1[$.  It is now well known (see \cite{Stefano15} and references therein) that  the pathwise uniqueness holds for this equation when $ p \in [\frac{1}{2}, 1[Ê$.  For $ 0 < p < \frac{1}{2} $, the boundary 0 is always  attainable (see \cite{AnersenPiterbarg07}) and a boundary condition at $V=0$ needs to be imposed to make the process unique. To stay in the most general setting, we consider weak solutions of (\ref{V_p}). We study the moment generating function of a transformation of $V$ and we will show that it explodes at a critical moment $ \mu^\ast$. We will then derive an asymptotic expansion for this MGF near $\mu^\ast$ and deduce an asymptotic formula for the cumulative distribution of $V_t$.

Let's first emphasize that when $p \in [\frac{1}{2}, 1[ $,  the (unique) solution admits a smooth density which decays exponentially as $ e^{ - \mu^\ast_t x^{ 2 (1-p)} } $ (see \cite{Stefano15} and \cite{Stefano11}) where $\mu^\ast_t$ is given by (\ref{critical-moment}). In particular, the moment generating function of $V^{2(1-p)}_t $ explodes for all moments larger than  $\mu^\ast_t$.  We draw attention  that in \cite{Stefano11} only the leading term $ e^{ - \mu^\ast_t x^{ 2 (1-p)} } $ of the density expansion is obtained. We will give a sharper asymptotic expansion for this density where the final result will be similar the one derived in Theorem~\ref{theorem-tauberian}.

\begin{theorem}\label{asymp_MGF}
Let $V_t$ be a weak solution to (\ref{V_p}), we have, for $x$ sufficiently large, 
\begin{equation}
\ln \mathbb{E} ~e^{( \mu^\ast_t - \frac{1}{x} ) V^\lambda_t } = \hat{\Delta} (t,x) + \mathcal{O} (x^\frac{-(2 - \lambda)}{\lambda} ),
\end{equation}
where $ \lambda = 2(1-p)$ and  $  \hat{\Delta} (t, x) $ is given by
\begin{eqnarray}\label{hat_delta}
 \hat{\Delta} (t, x)   
&=&
v^\lambda_0  e^{-b \lambda t}Ê{\mu^\ast}(t)^{2}  x  + \frac{\lambda - 1}{\lambda} \ln (  \mu^\ast_t  x) - v^\lambda_0  e^{- b \lambda t}Ê\mu^\ast_t   +
\nonumber\\ && 
v^\lambda_0 \sum_{i \geq 1} \frac{ \nu_i }{ 1 - i \frac{2 - \lambda}{\lambda} }   \omega^i_{ e^{b (\lambda - 1) t} } (  e^{-b \lambda t}Ê\mu^\ast_t (\mu^\ast_t x -1)  )^{ 1 - i \frac{2 - \lambda}{\lambda} } ,
\end{eqnarray}
with
\[
 \omega(\tau) = \frac{a \lambda}{ b( \lambda - 1) v_0}  (\frac{2b}{\sigma^2 \lambda}Ê)^\frac{2-\lambda  }{\lambda}     \int_1^\tau ( 1 -   s^\frac{-\lambda}{\lambda - 1} )^\frac{\lambda -2 }{\lambda}
  ds,
 \]
 and
\begin{equation}
\left\{ 
\begin{array}{ll}
 \nu_1 =     \frac{2( \lambda-1)}{\lambda}, ~~~\nu_{i+1} =\frac{\nu_1 }{i+1} ( 1 - (i+1) \frac{2 - \lambda}{\lambda})  \sum_{ k =1}^i C_k (\lambda) \sum_{ \substack{j_1, \dots, j_k   :\\ j_1 + \dots, j_k = i}}
   \nu_{j_1}  \dots  \nu_{j_k}   , \\
 C_k (\lambda) = \frac{   \frac{\lambda -1}{\lambda}  (  \frac{\lambda -1}{\lambda}  -1) \dots (   \frac{\lambda -1}{\lambda}  - k + 1  )     }{k!}  ,
  \end{array}
 \right.
\end{equation}
whenever $ x > \omega^\frac{\lambda}{2 - \lambda}_\tau  $
    \end{theorem}
 \begin{proof}
   The proof is in line with Theorem~\ref{uniqueness-theorem}. First using Proposition~\ref{proposition-tauberian} we have $\tilde{B} (t, x) = B(x) (1 + \mathcal{O}  (x^{-1}) )$. Next we use the sequence $ \tilde{R}_n$ defined in Theorem~\ref{uniqueness-theorem}; we find that 
   \[
    \tilde{R}_0 (\tau, y) = c_1 y^\frac{2(\lambda-1)}{\lambda} , ~~~~
   \tilde{R}_1 (\tau, y) = \tilde{R}_0 (t, y) + c_2 y^{ 1-\frac{2-\lambda}{\lambda} } + \mathcal{O} (y^{ 1-2\frac{2-\lambda}{\lambda} } ).
   \]
   The next iteration will give $\tilde{R}_2 (t, y) = \tilde{R_1 } (t, y) +  \mathcal{O} (y^{ 1-2\frac{2-\lambda}{\lambda} } )$, and so on \dots
      \end{proof}

\begin{remark}
The first statement of Theorem~\ref{asymp_MGF} tells us  that the moment generating function of the CIR process $X^{CIR,l}$ (which coincides with $X$ when $a=0$) is a good approximation of the moment generating function of $X$ near $\mu^\ast_t$ (hence the cumulative distribution of $X^{CIR,l}$ approximates $X$'s by the Tauberian theorem~\ref{theorem-tauberian}). The error of this approximation $R(t, x)$  is of order $ \mathcal{O} ( x^\frac{2(\lambda - 1)}{\lambda} )$ which is reasonably small for $\lambda < 1$ and allows us to derive an expansion for the cumulative distribution as
\[
 \mathbb{P} (X>x) = e^{ - \Lambda^\ast (x) }  ~x^\frac{-3}{4} (  c(t) + o (1)).
 \]
 Solving the equation for $R$ will allow obtaining two extra terms as in Theorem~\ref{theorem-tauberian}. In the other case (i.e. $\lambda > 1$) solving the equation on $R$ is needed to get a good expansion for the MGF and the cumulative distribution as the "extra" term $R$ is relatively large. 
  \end{remark}
  
\begin{corollary}  
Let $V_t$ be a weak solution to (\ref{V_p}). We have
\begin{equation}
\mathbb{P} (V^\lambda_t \geq x) = e^{ - \hat{\Lambda}^\ast (x)}  \left( \frac{ \sqrt{ { p^\ast}' (x) } }{ p^\ast (x) \sqrt{2 \pi}} - \frac{ 2 + \frac{\alpha}{ (\alpha + 1)^2 } }{  24 \mu^\ast  \sqrt{2 \pi} } ~\frac{1}{x^2 \sqrt{ {p^\ast}'(x) }  }    + o(    x^\frac{-5}{4} + x^\frac{-(2-\lambda)}{2 \lambda}      )   \right),
\end{equation}
 where $\hat{\Lambda}^\ast $ is the Fenchel-Legendre transform of $  \hat{\Lambda} (p) = \hat{\Delta} (t, \frac{1}{\mu^\ast -p}) $, where $ \hat{\Delta} $ is defined by (\ref{hat_delta}) and $p^\ast(x) =\partial_x \hat{\Lambda}^\ast  (x)$
  \end{corollary}
   
   \begin{proof}
   The proof follows from a direct application of Theorem~\ref{theorem-tauberian}.
   \end{proof}

\appendix

\section{The inverse function theorem}
\noindent Let us first recall briefly some facts from the theory of inverse and implicit functions in Banach spaces. 
\begin{definition}[Fr\'echet derivative]
Let $V$ and $W$ be normed spaces, and $ U   $  be a non-empty open subset of $V$. A function $f$ is called FrŽchet differentiable at $ x \in U$ if there exists a bounded linear operator $ df[a] : ~  V \longrightarrow W $ such that
\[
\lim_{ h \to 0} \frac{ f(a+h) - f(a) - df[a] h  }{\|ÊhÊ\|} = 0.
\]
\end{definition}

\begin{definition}[contraction and fixed point]
 Let $ (M, d)$ be a metric space. A mapping $f:~M \longrightarrow M$ is a contraction if there exists a real number $ k \in [0,1[$ such that
\[
d (f(x), f(y)) \le k d (x,y) , ~~~\text{for all } ~x, y \in M. 
\]
A fixed point of $f$ is $ x \in M$ such  that $ f(x) = x$. 
\end{definition}

\begin{theorem}[Contraction mapping]
If $ f$ is a contraction on a complete metric space $ (M, d)$, then $ f$ has precisely one fixed point in $M$.
\end{theorem}

\subsection{Proof of Lemma~\ref{banach-lemma}}
We can easily check that $(X, \|Ê\|) $ is a normed  vector space.  We only need to prove that $A_\infty(T, M, \gamma)$ is complete with respect to  the norm (\ref{normX}). Let $ (f_n) $  be a Cauchy-Sequence in $ A_\infty(T, M, \gamma)$. From the definition of $ A_\infty(T, M, \gamma)$,  for each $k \geq 0$ the sequence $ g^k_n := x^{\gamma+k} \partial^k f_n (t,x)$ is  Cauchy-Sequence in the space $ C_b ( [0,T] \times [M, \infty[ )$ of bounded continuous functions. As $ (C_b ( [0,T] \times [M, \infty[ ), \|Ê\|_\infty )$ is a Banach space,  the Cauchy-sequence $ g^{k}_n$ converges to $g^{k}$ in $ C_b ( [0,T] \times [M, \infty[ )$. On the other hand, noticing that $ \partial_x g^{k}_n = \frac{\gamma + k}{x} g^{k}_n + \frac{1}{x} g^{k + 1}_n  $, we deduce that $ \partial_x g^{k}_n$ converges uniformly to $ \frac{\gamma + k}{x} g^{k} + \frac{1}{x} g^{k + 1}$ in $ C_b ( [0,T] \times [x_0, \infty[ )$. Iterating this many times we can easily prove that $ g^k = x^{ \gamma + k} \partial^k f$, where $ f = x^{ - \gamma} g^0$. It follows that for each $ k \geq 0$, $ \sup_{ (t, x) \in [0,T] \times [M, \infty[  } |   x^{\gamma + k} \partial_k( f_n (t,k) - f (t,k)) |$ converges uniformly to 0. We conclude using the Banach space $\mathcal{L}^\infty (\mathbb{N})$ of all bounded sequences, with norm $ \|Ê\| = \sum_{k \geq 0} | x_k|$ that $ f_n$ converges to $f$ in $A_\infty(T, M, \gamma)$

\subsection{Proof of Theorem~\ref{uniqueness-theorem}}
As  $ \tilde{B} \in A_\infty (T, M; \beta) $  we have, for any $ s \le T, ~~ z > M$,
\begin{equation}\label{derivative-tilde-B-bound}
| \frac{ \partial^k \tilde{B} (s,z)}{ \partial x^k} | \le \|  \tilde{B}Ê \| z^{\beta - k}. 
\end{equation}
Similarly if $ u \in A_\infty (T, M; \gamma)$, where $ \gamma > 1 \vee \frac{1}{1 - \beta} $, then we have $ ( s, x) \mapsto  \tilde{B}  (  s, \xi_s ( x + \mu^\ast_s )^2 ( x_0 + \frac{\partial_u }{\partial_x} (s,x)Ê)) ~ \in A_\infty ( T, M ; \beta ( \gamma + 1))$.   Hence the operator 
\[
L: ~u \mapsto L (u) (t, x) :=  \int_0^t  \frac{ (x + 2 b ) \mu^\ast_sÊ}{  x+ \mu^\ast_s     } \tilde{B} \left[ s,\xi_s  (x+  \mu^\ast_s   )^2  (x_0 +  \partial_x u (s,x)    ) \right] d s 
\]
maps $A_\infty ( T, M ; \gamma ))$ into $A_\infty ( T, M ; \beta ( \gamma + 1))$. We emphasize that the hypothesis $ \gamma > 1 \vee \frac{ \beta}{1- \beta}$ implies that $ \beta ( \gamma + 1) < \gamma$. In particular, we can easily prove that $ L$ maps the unit ball of $ X_\gamma$ into itself for $M$ sufficiently large. Furthermore,  $ L: ~X_\gamma \longrightarrow X_\gamma$ is differentiable with
\[
dL[u] (h) (t,x) =  \int_0^t    D (s,x) h(s,x) d s,
\]
where
\[
D(s,x) := \xi_s  (x+  \mu^\ast_s   )^2  \frac{ (x + 2 b ) \mu^\ast_sÊ}{  x+ \mu^\ast_s     }  \partial_z \tilde{B} \left[ \xi_s  (x+  \mu^\ast_s   )^2  (x_0 +  \partial_x u (s,x)  ) \right] .
\]
In particular, using (\ref{derivative-tilde-B-bound}) we  have $ D \in A_\infty ( T, M ; \beta ( \gamma + 1) + 1 - \gamma) $.  We derive the following upper bound for  $|x^{- \gamma} dL[u] (h) (t,x) |$:
\[
|x^{- \gamma} dL[u] (h) (t,x) |  \le  x^{ - 1} \sup_{ s \le T }  |D(s,x)| ~ \sup_{ s \le T }  x^{ - \gamma + 1} |\partial_x h (s,x) | \le   x^{ - 1} \sup_{ s \le T }  |D(s,x)|  \|ÊhÊ\|,   
\]
We emphasize that the assumption $ \gamma > 1 \vee \frac{ \beta}{1- \beta}$ implies that  $\beta ( \gamma + 1) + 1 - \gamma < 1$. In particular $x^{ - 1} \sup_{ s \le T }  D(s,x) \to 0 $ as $x \to \infty$. Hence, for $ M$ sufficiently large we have
\[
\|Êd L[u] (h) Ê\| \le \kappa \|ÊhÊ\|, ~~~~\textrm{where}~~~ \kappa < 1.
\]
It follows that 
\[
\| ÊdL[u]  ÊÊÊ\| = \sup_{  \|ÊÊÊh\| = 1 } \|ÊÊdL[u]  (h) ÊÊ\| \le k < 1,
\]
so that for every $ u_1, u_2 \in  \bar{S}_1 (0) $ (the unit ball of $A_\infty (T, M; \gamma)$) we have
\[
 \|Ê  L(u_1) - L(u_2)   \| \le \sup_{ u\in \bar{B}_1 (0) } \|ÊÊd L[u]Ê\|  \|ÊÊu_1 - u_2ÊÊ\| \le k \|ÊÊu_1 - u_2\|.
\]
That is $L$ is a contraction on $\bar{S}_1 (0)$. Thus, $L$ has precisely one fixed point on $\bar{S}_1 (0)$. Now we can prove that if $ u \in A_\infty ( T, M ; \gamma) $ then $ u $ belongs to the unit ball of $ u \in A_\infty ( T, M ; \gamma + \epsilon) $, with $ \epsilon > 0$. Thus   (\ref{R_eq}) has precisely one solution on $ A_\infty ( T, M ; \gamma) $.

Now assume $\tilde{B} (x) = c x^\beta$.  Equation (\ref{R_eq}) takes the form:
\[
R(t, x) = c x^{2 \beta}_0  (x + 2 \frac{2 b }{\sigma^2}Ê)^{ 2 \beta} \int_0^t ~ e^{ b \beta s} { \mu^\ast_t }^{ 1 - 2 \beta} 
( 1 - \frac{e^{- b s}Ê\mu^\ast_s}{ x + \mu^\ast_s }Ê) ( 1 + \frac{2 b }{\sigma^2 x + 2 b }  \frac{e^{- b s}}{1 - e^{- b s} }Ê)^{ 2 \beta} ( 1 + \frac{1}{x_0} \partial_x R(s, x)Ê)^\beta d s.
\]
Hence using the variable substitution $ u = e^{ b \beta  s} $  we find that the function $\tilde{R} (\tau,y) := \frac{1}{x_0} R ( \frac{\ln(\tau)}{b \beta},  y - \frac{2b}{\sigma^2 }    )  $ will be given as follows:
\[
\tilde{R} (\tau, y ) = y^{ 2 \beta} \int_1^\tau \omega'_s ( 1 + \partial_x \tilde{R}(s, x)Ê)^\beta d s ~ \left( 1 +  \mathcal{O} ( y^{-1}) \right),
\]
where $ \omega$ is given by
\begin{equation}
\omega_\tau = c \frac{ x^{2 \beta}_0 }{ b \beta }Ê (\frac{2 b}{\sigma^2}Ê)^{ 1 - 2 \beta} \int_1^\tau  ( 1 - s^\frac{-1}{\beta}Ê)^{ 2 \beta - 1} d s
\end{equation}
Consider the sequence $ \tilde{R}_n$ defined by 
\[
 \tilde{R}_0 (t, y) = \eta_t y^\gamma, ~~~~~\tilde{R}_{ n+1} (t, y) = y^{ 2 \gamma}  \int_1^\tau \omega'_s ( 1 + \partial_x \tilde{R}_n (s, x)Ê)^\beta d s,
 \]
where 
\[
\eta_t = \omega_t 1_{ \beta < \frac{1}{2}   } +    (( 1 + \frac{1}{2}  \omega_t  )^2 - 1) 1_{\beta = \frac{1}{2}Ê}  +  ( \gamma^\beta (1 - \beta) )^{ \frac{1}{1-\beta} 1 Ê} 1_{ \beta> \frac{1}{2} } \omega^\frac{1}{1 - \beta}  1_{ \beta > \frac{1}{2}Ê}
\]
We find that $ \tilde{R}_1 (t, y) \sim \tilde{R}_0 (t, y)$, and hence for every $ n \geq 1$, $ \tilde{R}_n (t, y) \sim \tilde{R}_0 (t, y)$. Thus,
\[
\tilde{R} (t, y) \sim c(t) (t) y^\gamma
\]
\hfill\(\Box\)

\section{Proof of the Tauberian results}

\subsection{Proof of Lemma~\ref{lemma-bounds}}
The second derivative of $ \Lambda (p) := \ln \mathbb{E} ( e^{ p X} 1_{ X > 0}) $ is given by
\[
\Lambda'' (p) = (\mathbb{E} ( e^{ p X} 1_{ X > 0}) )^{ -2} \left[ \mathbb{E} ( X^2 e^{ p X} 1_{ X > 0}) \mathbb{E} (   e^{ p X} 1_{ X > 0}) - ( \mathbb{E} ( X e^{ p X} 1_{ X > 0}))^2 \right] 
\]
Now writing $X e^{ p X}1_{ X > 0}$ as $X e^{ p/2  X} ~e^{ p/2  X} 1_{ X > 0}$ we have using H\"older inequality
\[ 
 \mathbb{E} ( X e^{ p X}1_{ X > 0})  \le \left(  \mathbb{E} ( X^2 e^{ p X}1_{ X > 0})  \right)^\frac{1}{2} \left(  \mathbb{E} (   e^{ p X}1_{ X > 0})  \right)^\frac{1}{2} .
 \]
 Hence
 \[
 \Lambda'' (p) >0.
 \]
 Thus $\Lambda$ is convex. In particular $ \Lambda^\ast (x) = p^\ast (x) x - \Lambda (p^\ast(x))$, where $p^\ast (x)$ is the unique solution to $ x = \Lambda' (p^\ast (x))$. Now since 
 \[
 \frac{c_1}{ (\mu^\ast - p)^{ \alpha_1} }  \le \Lambda(p) \le \frac{c_2}{ (\mu^\ast - p)^{ \alpha_2} },
\]
we have
\[
\mu^\ast - \tilde{c}_2 x^\frac{\alpha_2}{\alpha_2 + 1}Ê = \sup_{p > 0} ( p x - \frac{c_1}{ (\mu^\ast - p)^{ \alpha_2} }) \le \Lambda^\ast (x) \le \sup_{p > 0} ( p x - \frac{c_2}{ (\mu^\ast - p)^{ \alpha_1} }) = \mu^\ast - \tilde{c}_1 x^\frac{\alpha_1}{\alpha_1 + 1},
\]
where $ \tilde{c}_i = (c_i\alpha_i )^\frac{1}{\alpha_i + 1}  ( 1 + \alpha_i)/\alpha_iÊÊ$. First from the monotonicity and the convexity of $\Lambda$ we have, $ \Lambda^\ast(x) = p^\ast(x) x - \Lambda(p^\ast(x))$, where $ p^\ast(x)$ is the unique solution to $ x = \Lambda'(p)$. Hence
\begin{equation}\label{bound-lmbda-start}
\tilde{c}_1 x^\frac{\alpha_1}{\alpha_1 + 1} \le (\mu^\ast - p^\ast(x) ) x + \Lambda( p^\ast(x))  \le \tilde{c}_2 x^\frac{\alpha_2}{\alpha_2 + 1}Ê.
\end{equation}
It follows that
\[
\mu^\ast - p^\ast(x)  \le \tilde{c}_2 x^\frac{-1}{\alpha_2 + 1}
\]
Let's now assume there exists $x > 0$ such that
\[
\mu^\ast - p^\ast(x)  <  (c_1 / \tilde{c}_2)^\frac{1}{\alpha_1} x^\frac{ - \alpha_2 / \alpha_1 }{1 + \alpha_2} .
\] 
For this $x$ we have
\[
 (\mu^\ast - p^\ast(x) ) x + \Lambda( p^\ast(x))  > \Lambda( p^\ast(x)) \geq \Lambda_1 (p^\ast (x)) 
 >
   \Lambda_1 ( \mu^\ast - (c_1 / \tilde{c}_2)^\frac{1}{\alpha_1} x^\frac{ - \alpha_2 / \alpha_1 }{1 + \alpha_2}   ) =\tilde{c}_2 x^\frac{\alpha_2}{\alpha_2 + 1} .
\]
This clearly  contradicts  (\ref{bound-lmbda-start}). Hence,
\[
\forall x > 0, ~~~~\mu^\ast - p^\ast(x)  \geq  (c_1 / \tilde{c}_2)^\frac{1}{\alpha_1} x^\frac{ - \alpha_2 / \alpha_1 }{1 + \alpha_2}.
\]
Thus
\begin{equation}\label{bound-pAst}
\forall x > 0, ~~~p^\ast (x) \in \left[    \mu^\ast - \tilde{c}_2 x^\frac{-1}{1 +\alpha_2 }, \mu^\ast - (c_1 / \tilde{c}_2)^\frac{1}{\alpha_1} x^\frac{ - \alpha_2 / \alpha_1 }{1 + \alpha_2} \right].
\end{equation}
We can easily see that for every $ p > 0$
\[
\frac{ d_1}{ ( \mu^\ast - p)^{\beta_1} }  \le \Lambda' (p) \le \frac{ d_2}{ ( \mu^\ast - p)^{\beta_2} } ,
\]
  where $ d_1, ~d_2 > 0$ and $ \beta_1$ and $\beta_2$ are given by
  \[
  \beta_1 = \frac{\alpha_1}{\alpha_2} ( 1 + \alpha_2)  , ~~~\textrm{and}~~~\beta_2 = 1 + \alpha_2.
  \] 
  Differentiating both sides of $(x = \Lambda' (p^\ast (x))$ we have
  \[
 {p^\ast}' (x) = 1/\Lambda'' (p^\ast (x)).
  \]
We will now derive similar results for $ \Lambda''$.  We consider the function $ h (x) = \sup_{ p \in [0, \mu^\ast[} \{Ê x p - \Lambda' (p)\}$. For every $ x > 0$, there exists $ p (x)$ (not necessarily unique as $ \Lambda'$ may not be convex) such that $ h   (x) = x p(x) - \Lambda'(p(x))$. By proceeding as for $ \Lambda$ we can prove that 
 $ p (x)  \in [ p_1 (x) , p_2 (x)Ê]$,  where
\[
p_1 (x) =  \mu^\ast - \tilde{d}_2 x^\frac{-1}{1 + \beta_2}      , ~~~\textrm{and}~~~~ p_2 (x) =  \mu^\ast - (d_1 / \tilde{d}_2)^\frac{1}{\beta_1} x^\frac{ - \beta_2 / \beta_1 }{1 + \beta_2} ,
\]
where $ \tilde{d}_i = (d_i\beta_i )^\frac{1}{\beta_i + 1}  ( 1 + \beta_1)/\beta_iÊÊ$. Let $ p^\ast (x) = \sup_{p \in [0, \mu^\ast[ } ( x = \Lambda'' (p) )$ and $ p_\ast (x) = \inf_{p \in [0, \mu^\ast[ } ( x = \Lambda'' (p) )$. In particular we have $[ p_\ast (x), p^\ast(x) ] \subset [ p_1 (x) , p_2 (x)Ê]$. On the other hand, we have  $ \Lambda'' (p) \le  x$ (resp. $ \Lambda'' (p) \geq  x$) whenever $ p \le p_\ast (x)$  (resp. $ p \geq p^\ast (x)$). Hence $ \Lambda'' (p_1 (x)) \geq x$ and $ \Lambda'' (p_2 (x)) \le x$. i.e.
\[
\forall x > 0, ~~~ \Lambda'' (\mu^\ast - \tilde{d}_2 x^\frac{-1}{1 + \beta_2} ) \le x
, ~~~
\textrm{and} ~~~ 
\Lambda'' (\mu^\ast - (d_1 / \tilde{d}_2)^\frac{1}{\beta_1} x^\frac{ - \beta_2 / \beta_1 }{1 + \beta_2} ) \geq x.
\] 
Thus,
\begin{equation}\label{bound-Lambda2}
\forall p > 0, ~~~\Lambda'' (p) \in \left[ 
( (\mu^\ast - p)   (\tilde{d}_2/d_1)^\frac{1}{\beta_1}      )^{-\frac{\beta_1}{\beta_2} (1 + \beta_2]}, 
 ( (\mu^\ast - p)/\tilde{d}_2)^{-[1 + \beta_2]} \right].
\end{equation}
 Using (\ref{bound-pAst}) and (\ref{bound-Lambda2}) and the fact that $ {p^\ast}' (x) = 1/\Lambda''(p^\ast (x))$, we derive the following bounds for $ {p^\ast}' (x)$:
 \[
 {p^\ast}' (x) \in \left[  
 m_1 x^{ - \frac{\alpha_2}{\alpha_1}( 1 +  \frac{1 }{1 + \alpha2} )Ê},
 m_2 x^{ - \frac{\alpha_1}{\alpha_2} (1 +  \frac{1 }{1 + \alpha2} )Ê},
 \right]
 \]
 with positive constant $ m_1$ and $ m_2$. Hence, if $ \alpha_2 ( 1 +  \frac{1 }{1 + \alpha2} ) < 2 \alpha_1 $ then $ \lim_{ x \to \infty} x^2 {p^\ast}' (x) = \infty$. \hfill\(\Box\)

 \subsection{Proof of Lemma~\ref{lemma-tauberian}}
\noindent Let's denote by
 \begin{eqnarray*}
 I (x) &=&   \int_{x^{\beta-1}}^\infty   g (x z) ~ e^{ \chi_x (z)} d z ,
 \end{eqnarray*}
 where
 \[
\chi_x (z) := (p^\ast(x) -p^\ast(zx) ) zx  +(\Lambda  (p^\ast(zx))- \Lambda  (p^\ast(x)) .
\]
The idea is to approximate $ \int_{x^{\beta-1}}^\infty g (x z)  ~ e^{ \chi_x (z)} d z $ by $ \int_{ 1 - h(x)}^{ 1 + h (x)}  g (x z)  ~ e^{ \chi_x (z)} d z $ for some $h$ satisfying \\ $\int_{x^{\beta-1}}^\infty g (x z)  ~ e^{ \chi_x (z)} 1_{ | z-1|>h} d z  \sim 0$.  For this we shall use the fact that ${\chi_x}' (z) = x( p^\ast (x) - p^\ast ( zx)) $, which means (as $ p^\ast$ is increasing which is a consequence of the convexity of $\Lambda$) that $\chi_x (.)$ is decreasing in $ ]1, \infty[$ and increasing on $ ]0, 1[$. Let's first use the following representation of $ \chi_x (z)$, where we mainly use the fact that $ {\Lambda^\ast}'' (x) = { p^\ast} (x)$, 
 \begin{eqnarray}\label{equi_chiX}
 \chi_x (z) &=& p^\ast (x) ( z x - x) + \Lambda^\ast ( x) - \Lambda^\ast (zx) = -\int_x^{zx} ( p^\ast (u) - p^\ast (x)) d z = - \int_x^{ zx } \int_x^u { p^\ast}' (u) d u
 \nonumber\\ &=& - x^2 \int_0^{ |z-1|} \int_0^u {p^\ast}' ( ( 1 + \sign(z-1) v ) x) d v d u \sim_{ z \to 1} - x^2 { p^\ast}' (x) \frac{ ( z-1)^2}{2}.
 \end{eqnarray}
On the other hand we have from Lemma~\ref{lemma-bounds}  
\[
{p^\ast}' (u) \geq c u^{-2 +\gamma} , ~~~Ê\textrm{where}~~~\gamma = \frac{\alpha_2}{\alpha_1} ( 1 + \frac{1}{1+ \alpha_2}Ê) -2 > 0.Ê
\] 
Let's choose $ h$ as:
\[
h \equiv h(x) := x^{ -1 +\frac{\gamma}{4}} / \sqrt{   {p^\ast}' (x)  }.  Ê
\]
In particular  $h(x) \le \frac{1}{c}  x^\frac{-\gamma}{2}Ê \to_{ x \to \infty} 0 $, and $ h^2 (x) x^2 { p^\ast}' (x)  = x^\frac{\gamma}{2}Ê\to_{ x \to \infty} \infty$. Now for $ z > 1+ h$, we have
\begin{eqnarray*}
e^{\chi_x (z) }  &= &  e^{\chi_x (1+h)  }~e^{ \chi_x (z) - \chi_x (1+h)     } = e^{\chi_x (1+h)  } ~Êe^{  -\int_{ (1+h)x }^{zx} ( p^\ast (u) - p^\ast (x)) d z  } 
\\
&\le& e^{\chi_x (1+h)  } ~Êe^{  -\int_{ (1+h)x }^{zx} ( p^\ast ((1+h) x) - p^\ast (x)) d z  } = e^{\chi_x (1+h)  } e^{ - x  [p^\ast ( (1+h)x) - p^\ast (x)]  (z-h)  }.
\end{eqnarray*}
It follows that for any function with (at most) polynomial growth $g$ we have
\[
\int_{1+h}^\infty e^{\chi_x (z) }  \le e^{ \chi_x (1+h) } ~\int_{ (1+h) }^\infty g(z) e^{ - x  [p^\ast ( (1+h)x) - p^\ast (x)]  (z-h)  } d z \le Q(x) e^{ \chi_x (1+h) } \le Q(x) e^{ - \frac{1}{2} x^\frac{\gamma}{2}ÊÊ},
\]
where $ Q$ is some polynomial function. A similar result may be easily obtained for $ \int_{ x^{- \beta} }^{ 1 - h}$ using the fact that $ \chi_x (.) $ is decreasing on $ ]0, 1[$. This proves that
\[
\int_{x^{\beta-1}}^\infty g (x z)   e^{ \chi_x (z)} d z ~ \sim~ \int_{ 1 - h(x)}^{ 1 + h (x)}  g (x z)  ~ e^{ \chi_x (z)} d z.
\]
The integral in the right-hand side satisfies:
\begin{eqnarray*}
\int_{ 1 - h(x)}^{ 1 + h (x)}  g (x z)  ~e^{ \chi_x (z)} \ d z  &\sim& g(x) \int_{ 1 - h(x)}^{ 1 + h (x)}     e^{ \chi_x (z)} d z
=
g(x) \int_0^h (e^{  \chi_x (1+z)    } + e^{  \chi_x (1-z)    }) d z 
\nonumber\\ &=& -g (x) 
(\int_0^{ \sqrt{- 2 \chi_x (1+h) } }
\frac{u ~e^{ - \frac{1}{2} u^2Ê} du }{   { \nu^+_x }'(  { \nu^+_x }^{-1} ( - \frac{u^2}{2}Ê)   )    } 
+\int_0^{ \sqrt{- 2 \chi_x (1-h) } }
\frac{u ~e^{ - \frac{1}{2} u^2Ê} du }{   { \nu^-_x }'(  { \nu^-_x }^{-1} ( - \frac{u^2}{2}Ê)   )    } ),
\end{eqnarray*}
where we have used the variable substitutions  $ u = \sqrt{ - 2 \nu^\pm_x (z)}$, with $ \nu^\pm_x (z) := \chi_x (1 \pm z)$. Now using (\ref{equi_chiX}) we have  
\[
\nu^\pm_x (z) = - x^2 \int_0^{ z} \int_0^u {p^\ast}' ( ( 1  \pm  v ) x) d v d u \sim_{ z \to 0} - x^2 { p^\ast}' (x) \frac{ z^2}{2},
\]
and
\[
{\nu^\pm_x}' (z) = - x^2  \int_0^z {p^\ast}' ( ( 1  \pm  v ) x) d v   \sim_{ z \to 0} - x^2 { p^\ast}' (x) z.
\]
In particular $ {\nu^\pm_x}^{-1} ( - \frac{u^2}{2}Ê) \sim \frac{u}{   x \sqrt{  {p^\ast}' (x)  }   }   $ and ${ \nu^+_x }'(  { \nu^+_x }^{-1} ( - \frac{u^2}{2}Ê)   )  \sim  -x \sqrt{  {p^\ast}' (x)  } ~u$. On the other hand, we have
\[
\sqrt{ -2 \chi_x (1 \pm h)}  \sim x \sqrt{  {p^\ast}' (x)  } h = x^\frac{\gamma}{4}  \to \infty.
\]
Thus,
\[
\int_{ 1 - h(x)}^{ 1 + h (x)}  g (x z)  ~e^{ \chi_x (z)} \ d z ~\sim_{ x \to \infty} ~ 2 g (x)  x \sqrt{  {p^\ast}' (x)  } \int_0^\infty e^{ - \frac{1}{2} u^2 } d u =\frac{ \sqrt{ 2 \pi }Êg (x) }{ x \sqrt{  {p^\ast}' (x)  } }.
\]

If $ x \mapsto \Lambda(\mu^\ast - \frac{1}{x}Ê) \in \bar{A}_3 (T, M; \alpha)$, then we find that $ {p^\ast}' \in \bar{A}_1 (T, M;  \frac{-(\alpha + 2)}{\alpha + 1}Ê)$ and  in particular that 
\[
 {p^\ast}'' (x) \sim \frac{-(\alpha + 2)}{\alpha + 1}Ê  {p^\ast}' (x).
\]
We then derive (\ref{taub_res-refined}) by applying first order Taylor expansion to $ {p^\ast}'$.
\hfill\(\Box\)

 \subsection{Proof of Theorem~\ref{theorem-tauberian}}
  We  prove inequalities (\ref{lim_sup2}) and (\ref{lim_inf2}) by contradiction argument. Let's suppose that there exists $c<0 $ such that 
\[
\limsup_{x \to \infty} Ê\left(   \ln \mathbb{P} (X \geq x) + \Lambda^\ast (x)  +\ln ( \frac{ p^\ast(x) \sqrt{2 \pi} }{  \sqrt{ {p^\ast}' (x) }  } ) \right) \le c < 0 .
\] 
 This will  mean that for $x$ sufficiently large, 
\[
  \mathbb{P} (Z \geq z)  \le   \frac{ e^c }{\sqrt{2 \pi} }  \frac{  \sqrt{ {p^\ast}' (x) }  }{ p^\ast(x) }~e^{ -\Lambda^\ast (z) }.
\]
whenever $z \geq x^\beta$, where $\beta  > 0$ is such that $ x^\beta \ll \Lambda ( p^\ast(x))$.Ê 
Now using the following representation for the exponential function $  e^{p X} = e^{ p (c \wedge X)} + p \int_c^\infty e^{ p z} 1_{ z \le X} d z $,  we have, choosing  $ c = x^\beta$,
 \begin{eqnarray*}
    e^{ \Lambda (p^\ast (x))}  \le  e^{p^\ast (x)  x^\beta}  +
   p^\ast(x) e^{ \Lambda(p^\ast(x)) } x  \int_{ x^{\beta-1} }^\infty  \frac{ e^c }{\sqrt{2 \pi} }  \frac{  \sqrt{ {p^\ast}' (zx) }  }{ p^\ast(zx) } ~ e^{  \chi_x (z)} d z.
 \end{eqnarray*}
Using Lemma~\ref{lemma-bounds}  we have
 \[
 p^\ast(x)  x  \int_{ x^{\beta-1} }^\infty  \frac{ e^c }{\sqrt{2 \pi} }  \frac{  \sqrt{ {p^\ast}' (zx) }  }{ p^\ast(zx) } ~ e^{  \chi_x (z)} d z ~ \sim~ p^\ast(x)  x      \frac{ \sqrt{ 2 \pi }Ê  (   \frac{ e^c }{\sqrt{2 \pi} }  \frac{  \sqrt{ {p^\ast}' (x) }  }{ p^\ast(x) }  ) }{ x \sqrt{  {p^\ast}' (x)  } } = e^c < 1 .
 \]
 This will mean that for $x$ sufficiently large,
 \[
  e^{ \Lambda (p^\ast (x))}  \le  e^{p^\ast (x)  x^\beta}  + \frac{e^c+1}{2}  e^{ \Lambda(p^\ast(x)) },
 \]
with $ x^\beta \ll \Lambda ( p^\ast(x))$ and $ \frac{e^c+1}{2} < 1$. This is obviously impossible and shows that our assumption can not be true. Thus
\[
\limsup_{x \to \infty} Ê\left( \ln ( \frac{ p^\ast(x) \sqrt{2 \pi} }{  \sqrt{ {p^\ast}' (x) }  } ) \Big[ \ln \mathbb{P} (X \geq x) + \Lambda^\ast (x)  \Big] \right) \geq 0 .
\] 
We prove the  "$\liminf$" statement in a similar way; we suppose there exists $c'$ such that
\[
\liminf_{x \to \infty} Ê\left(   \ln \mathbb{P} (X \geq x) + \Lambda^\ast (x)  +\ln ( \frac{ p^\ast(x) \sqrt{2 \pi} }{  \sqrt{ {p^\ast}' (x) }  } ) \right) \geq c' > 0 ,
\] 
which will mean that for $x$ sufficiently large, 
\[
  \mathbb{P} (Z \geq z)  \geq   \frac{ e^{c' }}{\sqrt{2 \pi} }  \frac{  \sqrt{ {p^\ast}' (x) }  }{ p^\ast(x) }~e^{ -\Lambda^\ast (z) },
\]
whenever $ z> x^\beta$. This will lead to the contradiction
\[
  e^{ \Lambda (p^\ast (x))}  \geq  e^{p^\ast (x)  x^\beta}  + \frac{e^{c'}+1}{2}  e^{ \Lambda(p^\ast(x)) },
 \]
with $ x^\beta \ll \Lambda ( p^\ast(x))$ and $ \frac{e^c+1}{2} >1 $.  

If $ x \mapsto \Lambda(\mu^\ast - \frac{1}{x}Ê) \in \bar{A}_3 (T, M; \alpha)$ then we use (\ref{taub_res-refined}). We prove in a samilar way that
\[
\limsup_{x \to \infty }  \left(  x^2 \sqrt{ {p^\ast}'(x)}   \left[
\ln~\mathbb{P} (X \geq X) + \Lambda^\ast (x) - \frac{  \sqrt{ {p^\ast}'(x)}   }{ p^\ast (x) \sqrt{ 2 \pi}}   \geq  -\frac{2 + \frac{\alpha}{(\alpha + 1)^2}Ê}{ 24 \mu^\ast \sqrt{2 \pi}}Ê
\right]
\right),
\]
and
\[
\liminf_{x \to \infty }  \left(  x^2 \sqrt{ {p^\ast}'(x)}   \left[
\ln~\mathbb{P} (X \geq X) + \Lambda^\ast (x) - \frac{  \sqrt{ {p^\ast}'(x)}   }{ p^\ast (x) \sqrt{ 2 \pi}}   \le  -\frac{2 + \frac{\alpha}{(\alpha + 1)^2}Ê}{ 24 \mu^\ast \sqrt{2 \pi}}Ê
\right]
\right).
\]
\hfill\(\Box\)

\subsection{Proof of Proposition~\ref{proposition-tauberian}}  
We shall not assume that $X$ admits a smooth density; if that would be the case we could express  $ \mathbb{E} (U e^{px})$ by integrating with respect to this density function. Instead, we proceed as in \cite{Aly13} by using  the following representation of the exponential function:
\begin{eqnarray*}
g(U)  e^{p U} &=& g(U) 1_{ U \le c}  e^{p U} +  1_{ U > c} g(U)  e^{p U}  \\
&=& g(U) 1_{ U \le c}  e^{p U} + 1_{ U > c} \left(  g(c)  e^{ pc} +   \int_c^U e^{p z} (g'(z)  + p g(z) )  d z  \right)
\\ &=&
 g(U) 1_{ U \le c}  e^{p U} + 1_{ U > c}    g(c)  e^{ pc} + \int_c^\infty e^{p z} (g'(z)  + p g(z) ) 1_{z \le U} d z
 \\ &=&
 g(U \wedge c)    e^{p~ U \wedge c} + \int_c^\infty e^{p z} (g'(z)  + p g(z) ) 1_{z \le U} d z,
\end{eqnarray*}
which holds for any $ U ,  p, c \in \mathbb{R}$. Applying this to $U= X$  we have
\begin{eqnarray}\label{reprentation_exponential}
\mathbb{E} [ g(X) e^{p X} ]&=&   \mathbb{E}  [g(X \wedge c)    e^{p~ X \wedge c}]     + \int_c^\infty e^{p z} ( g'(z) + p g(z)) \mathbb{P}(z \le X) d z.
\end{eqnarray}
Let's denote, for a function $h$ with polynomial growth, 
\[
I^h (p) :=  \int_c^\infty  h(z) e^{p z}~ \mathbb{P}(z\le X ) d z.
\]
With this notation we may write $\mathbb{E} [ g(X) e^{p X} ] $ as
\begin{equation}\label{I_zeta}
\mathbb{E} [ g(X) e^{p X} ]  =   E_g(c)   +   I^{ g' } (p) + p I^{g  } (p),
\end{equation}
where 
\[
E_g (c) =   \mathbb{E}  [g(X \wedge c)    e^{p~ X \wedge c}]    .
\]
The idea here is to choose $c$ to be large but so that $ E_g (c)$ is negligible with respect to the two terms $ I^{ g' } $  and $I^{ g } $. It is worth emphasizing that under the assumption $ \mathbb{E} | g (X) |=  m \le \infty$, we have
\begin{equation}\label{uper_bound_Rg}
|E_g (c)| \le  \mathbb{E}  |g(X)|     e^{p c}  +  g (c) e^{Êpc} = ( m +  g (c) )e^{Êpc}.
\end{equation}
Now applying  Theorem~\ref{theorem-tauberian} to $X$ we have, for $ z$ sufficiently large,
  \[
  \mathbb{P}(z\le X ) ~ \sim~ e^{ -\Lambda^\ast (z) } ~\frac{ \sqrt{ÊÊ{p^\ast }'(z)Ê} }{ p^\ast (z) \sqrt{2 \pi}}.
  \]
For $x$ sufficiently large, define $ Z \equiv Z(x)$ by
 \begin{equation}\label{Z_x}
 Z  = \Lambda' (t, \mu^\ast - \frac{1}{x} ) 
  \end{equation}
In particular, we have 
\[
 \mu^\ast - \frac{1}{x} = p^\ast  (Z).
\]
 It follows that, for any function $h$ (for which $I^h$ is well defined),   $ I^h ( \mu^\ast  -\frac{1}{x} ) \equiv I^h ( p^\ast (Z)) $ may be written as
\[
I^h( p^\ast (Z)) = \int_{c}^\infty h(z) e^{p^\ast (Z) } e^{ \Lambda^\ast (z) } f(z) d z = Z e^{\Lambda( p^\ast (Z))}
\int_{c/Z}^\infty  h(zZ)  ~ \frac{ \sqrt{ÊÊ{p^\ast }'(zZ)Ê} }{ p^\ast (zZ) \sqrt{2 \pi}}~ e^{ \chi_x (z)} d z .
\]
We then choose $ c = Z^\beta$,  where $\beta $  is such that $ x^\beta \ll \Lambda (p^\ast (x)) $.  Lemma~\ref{lemma-tauberian} ensures that
\[
\int_{c/Z}^\infty  h(zZ)  ~ \frac{ \sqrt{ÊÊ{p^\ast }'(zZ)Ê} }{ p^\ast (zZ) \sqrt{2 \pi}}~ e^{ \chi_x (z)} d z  ~ \sim~  h (Z)\frac{ \sqrt{ÊÊ{p^\ast }'(Z)Ê} }{ p^\ast (Z) \sqrt{2 \pi}} \frac{ \sqrt{ 2 \pi}  }{  Z  \sqrt{   {p^\ast}' (Z)  }} = \frac{h (Z)}{ Z p^\ast (Z)} 
\]
 It follows that for $x$ sufficiently large we have
\begin{equation}\label{ratio_Is}
\frac{I^g }{I^{ z \mapsto 1}} ( p^\ast  (Z))  ~ \sim ~ \frac{g(Z) }{ p^\ast (Z)} , 
\end{equation}
A similar statement holds for $ I^{g'} (p^\ast (Z))$. The notation $ z \mapsto 1$ refers to the "constant" function mapping $ \mathbb{R}$ into $ \{1\}$. Now using the fact that $\beta$ is such that $ Z^\beta$ is negligible with respect to $ \Lambda (p^\ast (Z)) $ we have
\begin{eqnarray*}
\frac{  \mathbb{E} [ g(X) e^{  p^\ast (Z) X} ]     }{  \mathbb{E} [   e^{  p^\ast (Z) X} ]   }  
&=& 
 \frac{ E_g ( Z^\beta)  +   I^{ g' } + p^\ast (Z) I^g   }{  E_{z \mapsto 1 }  (Z^\beta) +  p^\ast (Z) I^{ z \mapsto 1}  } 
 \\ &=&
  \frac{ 1}{  p^\ast (Z)}  \frac{I^{g'} }{I^{ z \mapsto 1}} 
+   \frac{I^{ g } }{I^{z \mapsto 1}} + o ( e^{Z^{\beta} - \Lambda (p^\ast (Z))  }) ,
\end{eqnarray*}
where we used (\ref{uper_bound_Rg}) which ensures that $ E_g (Z^\beta)| \le ( m + g(Z^\beta) )  e^{ pZ^\beta } $.  In particular $ \frac{R_g (Z^\beta) }{ I^{ z \mapsto 1} }Ê= o ( e^{p Z^{\beta} - \Lambda (p^\ast (Z))  })$ decays exponentially to 0 as $Z \to infty$.   
 Hence
\[
\frac{  \mathbb{E} [ g(X) e^{ p^\ast (Z) X} ]     }{  \mathbb{E} [   e^{ p^\ast (Z) X} ]   }   ~ \sim~ g(Z) + \frac{1}{ p^\ast(Z)} g' (Z) .  
\]
The  proof is completed by replacing $Z$ by $ \Lambda' (\mu^\ast - \frac{1}{x}) $, which comes from the definition of $Z$ being solution to $ x= (\mu^\ast - p^\ast (Z) )^{-1}$ and the fact that $ p^\ast (Z) $ is defined as: $ \Lambda' (p^\ast (Z)) = Z$.
 
The refined statement is proved by using the refined statement of Theorem~\ref{theorem-tauberian}. \hfill\(\Box\)

\section{Proof of Theorem~\ref{moment-explosion}}\label{poof-ME}  
\noindent Let's first consider the case where  $ y \mapsto \frac{B(y)}{ y}Ê$ is increasing  on $ [M, \infty[$. Then for any  $ y \geq M$ we have
\[
\frac{B(y)}{ y} \le -b = \lim_{ x \to \infty } \frac{B(x)}{ x}.
\]
Hence if we denote  $ m := \sup_{ y \le M} (y) $ we have
\[
\forall y \geq 0, ~~~~ B (y) \le m - b y.
\] 
On the other hand for any $ Z \geq M$ and $y > 0$ we have, using the fact that $ \frac{B( Z)}{  Z }   \le \frac{B(y)}{ y}$ and the definition of $ \bar{B} (y) := B(y) + b y$ (notice that the fact that $ y \mapsto \frac{B(y)}{ y}Ê$ is increasing towards $(-b)$ will mean that $ y \mapsto \frac{\bar{B}(y)}{ y}Ê$ is increasing towards $0$, hence negative. This will, in  particular, mean that there exists $ c > 0$ such that $ \bar{B} (y) \geq - c y$. Furthermore if $ \bar{B} \in A_0 (M; \beta)$ then there exists $ c > 0$ such that $ \bar{B} (y) \geq - c y^\beta$),
\[
B(y) - \frac{B( Z)}{  Z } y =  - b y + \bar{B} (y) - \frac{-b Z + \bar{B}( Z)}{  Z } y  =    \bar{B} (y)   -  \frac{\bar{B} (Z) }{Z}    y  \geq  \bar{B} (y) 1_{ y \le Z} \geq - c Z^\beta,
\]
where we used the fact that for $ y \geq Z$, $   \frac{\bar{B} (Z) }{Z} \le  \frac{\bar{B} (y) }{y} $, and hence $ \bar{B} (y)   -  \frac{\bar{B} (Z) }{Z}    y \geq 0$. Thus,
\begin{equation}\label{bound2_increas}
\forall y \geq 0, ~~~  c  - c Z^\beta + \frac{B(Z)}{ Z} y \le B(y) \le m -b y
\end{equation}
Similarly, if $ y \mapsto \frac{B(y)}{ y}Ê$ is decreasing  on $ [M, \infty[$ we find
 \begin{equation}\label{bound2_decreas}
  \forall y \geq 0, ~~~  c Z^\beta + \frac{B(Z)}{ Z} y \geq B(y) \geq -m -b y
\end{equation}

By the comparison theorem (see e.g.  \cite{Ikeda77}) the process $ X$ is bounded from below and above by two Cox-Ingersoll-Ross (CIR) processes $ V^{ \pm c Z^\beta , - \frac{B(Z)}{  Z} , X_0 }$ and  $ V^{  \mp m ,  b, X_o}$, where $ V^{a,\kappa, X_0}$ is the unique solution to the SDE
\[
d V_t = ( a - \kappa V_t ) d t + \sqrt{V_t} d W_t, ~~~~Ê  V_0 = X_0.
\] 
The moment generating function of $V_t $ is well known (see eg. \cite{Aly13}).  It  is finite for any $ \mu < \mu^\ast_t (\kappa) :=\frac{2 \kappa}{   1 - e^{ - \kappa t} }  $ and satisfies, for $t >0$ and $x$ sufficiently large,
\begin{eqnarray}\label{asym_cir}
  \ln \mathbb{E} ~e^{( \mu^\ast_t - \frac{1}{x})ÊV_t} &=& v_0 e^{ - \kappa t } { \mu^\ast_t}^2 x + 2a   \ln ({ \mu^\ast_t} x ) 
 - v_0  e^{ - \kappa  t} \mu^\ast_t .
\end{eqnarray}
For $Z$ sufficiently large,  $ \mu^\ast_t ( \frac{B( Z)}{  Z})$ is given in terms of $ \mu^\ast_t (b)$ as
\begin{eqnarray}\label{expansion-critical-moment}
\mu^\ast_t (- \frac{B( Z)}{  Z}) &=& \frac{2(  b -  \frac{\bar{B}( Z)}{ Z} )}{ 1 - e^{   - (  b -  \frac{\bar{B}( Z)}{ Z})t }} 
\\ &:=&
\mu^\ast_t (b) - \frac{1}{b} ( 1 - \frac{ b t e^{-bt}  }{1 - e^{-bt}} ) \mu^\ast_t (b) \frac{\bar{B}( Z)}{ Z} + \mathcal{O} ((\frac{\bar{B}( Z)}{ Z})^2 ).
\end{eqnarray}
On the other hand, if we denote
\begin{equation}\label{mgf_CIRr}
\Lambda (a, -\kappa ) (t,x) := \ln \mathbb{E} e^{( \mu^\ast_t  - \frac{1}{x} )Ê V^{ a , \kappa, X_0 }_t },
\end{equation}
we have
\begin{eqnarray*}
\Lambda (  \pm c Z^\beta, - \frac{B(Z)}{  Z}) (t, x)   
 &=&
 \ln \mathbb{E} ~ e^{ (\mu^\ast_t(- \frac{B( Z)}{  Z})-  ( \frac{1}{x} +\mu^\ast_t (- \frac{B( Z)}{  Z}) - \mu^\ast_t(b) Ê) ) V^{ a , - \kappa, X_0 }_t }
 \nonumber\\ &=&
  X_0 e^{ - ( b  -  \frac{\bar{B}( Z)}{ Z}) t } { \mu^\ast_t (- \frac{B( Z)}{  Z} )}^2  ( \frac{x}{  1Ê-x(\mu^\ast_t- \mu^\ast_t (- \frac{B( Z)}{  Z}) ) })  +  \nonumber\\ &&
   \pm 2 c Z^\beta \ln ( \mu^\ast_t (- \frac{B( Z)}{  Z} ) \frac{x}{  1Ê-x(\mu^\ast_t - \mu^\ast_t (- \frac{B( Z)}{  Z} ) ) })  - X_0  e^{  -   ( b  -  \frac{\bar{B}( Z)}{ Z}) t   } \mu^\ast_t (- \frac{B( Z)}{  Z} ) .
\end{eqnarray*}
On the other hand, we have
\begin{equation}\label{bound_CIR}
\Lambda ( \pm m, b) (t,x) := \ln \mathbb{E} ~e^{ ( \mu^\ast_t - \frac{1}{x})ÊV^{\pm m, b, X_0}_t } = X_0 e^{ - \kappa t } { \mu^\ast_t}^2 x \pm  2m   \ln ({ \mu^\ast_t} x ) 
 - X_0  e^{ - \kappa  t} \mu^\ast_t.
 \end{equation}

Let's now treat the cases $ y \mapsto \frac{B(y)}{ y}Ê$  decreasing and increasing separately again; In the first case, we have
\[
            V^{    -m ,  b, X_o}_t  \le X_t  \le V^{ c Z^\beta , - \frac{B(Z)}{  Z} , X_0 }_t  .
\]
In particular we have
\begin{equation}
 \Lambda ( -m, b) (t,x) \le  \ln \mathbb{E}   e^{( \mu^\ast_t - \frac{1}{x}Ê) X_t} \le   \Lambda (  c Z^\beta, - \frac{B(Z)}{  Z}) (t, x).
\end{equation}
Let's choose $ Z \equiv Z(x) = (x \log(x))^\frac{1}{1-\beta} $ (we need to have  $   \frac{\bar{B}( Z)}{ Z}  \ll \frac{1}{x} $ in order to be able to write $\frac{x}{  1Ê-x(\mu^\ast_t - \mu^\ast_t (- \frac{B( Z)}{  Z} ) ) }$ as $1 + x(\mu^\ast_t - \mu^\ast_t (- \frac{B( Z)}{  Z} )$). We have using (\ref{expansion-critical-moment})
\begin{eqnarray*}
\Lambda (  c Z^\beta, - \frac{B(Z)}{  Z}) (t,x) &=& X_0 e^{ - ( b  -  \frac{\bar{B}( Z)}{ Z}) t } { \mu^\ast_t (b )}^2 x ( 1 +  \frac{1}{b} ( 1 - \frac{ b t e^{-bt}  }{1 - e^{-bt}} ) (x \mu^\ast_t (b)- 2) \frac{\bar{B}( Z)}{ Z}  + \mathcal{O } (( x \frac{\bar{B}( Z)}{ Z} )^2)   ) 
\\ &&
 2 c Z^\beta \ln ( \mu^\ast_t (b ) x )  + Z \mathcal{ O} ( x \frac{\bar{B}( Z)}{Z}).
 \end{eqnarray*}
Hence for any  $t> 0$ there exists $ c_t > 0$ such that
\begin{equation}\label{upper-bound-MGF-Decreasing}
\Lambda (  c Z^\beta, - \frac{B(Z)}{  Z}) (t,x) \le X_0 e^{ - b   t } { \mu^\ast_t (b )}^2 x + c_t x^\frac{\beta}{1-\beta}  \log(x)^\frac{1}{1 - \beta}.
\end{equation}
The statement of theorem (for the decreasing case) follows immediately from (\ref{bound_CIR}) and (\ref{upper-bound-MGF-Decreasing}).

When   $ y \mapsto \frac{B(y)}{ y}Ê$  increasing,  we have
\[
    V^{ -c Z , - \frac{B(Z)}{  Z} , X_0 }_t  \le X_t  \le        V^{    m ,  b, x_o}_t   
\]
In particular we have
\begin{equation}
 \Lambda (  -c Z^\beta, - \frac{B(Z)}{  Z}) (t, x) \le  \ln \mathbb{E}   e^{( \mu^\ast_t - \frac{1}{x}Ê) X_t} \le   \Lambda ( m, b) (t, x)   
\end{equation}
The right-hand side of the inequality shows that $\ln \mathbb{E}   e^{( \mu^\ast_t - \frac{1}{x}Ê) X_t}  $ is finite for any $x$ sufficiently large. On the other hand, $\Lambda (  -c Z^\beta, - \frac{B(Z)}{  Z}) (t, x)$ is given as (choosing $ Z = (x \log(x))^\frac{1}{1-\beta} $ as in the decreasing case)
\begin{eqnarray*}
\Lambda ( - c Z^\beta, - \frac{B(Z)}{  Z}) (t, x) &=& X_0 e^{ - ( b  -  \frac{\bar{B}( Z)}{ Z}) t } { \mu^\ast_t (b )}^2 x + h(t) x^2  \frac{\bar{B}( Z)}{ Z}    - 2 c Z^\beta \ln ( \mu^\ast_t (b ) x )  + Z \mathcal{ O} ( x \frac{\bar{B}( Z)}{Z})
\\
&  \geq & X_0 e^{ - b   t } { \mu^\ast_t (b )}^2 x - h(t) \frac{x}{\log(x)} - 2 c x^\frac{\beta}{1 - \beta} \ln(x)^\frac{1}{1-\beta}, 
\end{eqnarray*}
where $ h(t) > 0$. If $ \beta < \frac{1}{2}Ê$ then $ x^\frac{\beta}{1 - \beta} \ll x$ and hence for any $ t > 0$ there exists $ \omega_1 (t) > 0$ such that
\begin{equation}\label{lower-bound-MGF-Increasing}
\Lambda ( - c Z^\beta, - \frac{B(Z)}{  Z}) (t, x) \geq \omega_1 (t)  x .
\end{equation}
 The statement for the increasing case follows  immediately from (\ref{bound_CIR}) and (\ref{lower-bound-MGF-Increasing}).
 \hfill\(\Box\)

\end{document}